\documentclass[reqno]{article}
\usepackage{amsmath,amsfonts,amssymb,amsthm}
\usepackage[final]{graphicx}

\usepackage{bbm}
\newtheorem{theorem}{Theorem}
\newtheorem{lemma}[theorem]{Lemma}
\newtheorem{corollary}[theorem]{Corollary}

\newtheorem{conjecture}[theorem]{Conjecture}

\newcommand\ceil[1]{\lceil #1 \rceil}
\newcommand\eps{\varepsilon}
\renewcommand\le{\leqslant}
\renewcommand\ge{\geqslant}
\renewcommand\varrho{\rho} 

\newcommand\NN{\mathbb{N}}
\newcommand\ZZ{{\mathbb Z}}
\newcommand\Rp{{\mathbb R^+}}

\newcommand\E{{\mathbb E}}
\renewcommand\Pr{{\mathbb P}}
\newcommand\pto{\overset{\mathrm{p}}{\to}}
\newcommand\dtv[2]{\mathrm{d}_{\mathrm{TV}}\bigl(#1\,,#2\bigr)}
\newcommand\Po{\mathrm{Po}}

\newcommand\nopf{\qed}

\newcommand\cB{{\mathcal B}}
\newcommand\cC{{\mathcal C}}
\newcommand\cD{{\mathcal D}}
\newcommand\cE{{\mathcal E}}
\newcommand\cL{{\mathcal L}}
\newcommand\cI{{\mathcal I}}
\newcommand\cM{{\mathcal M}}
\newcommand\cQ{{\mathcal Q}}
\newcommand\cR{{\mathcal R}}
\newcommand\cS{{\mathcal S}}
\newcommand\cT{{\mathcal T}}

\newcommand{\uv}{{\underline v}}
\newcommand{\uu}{{\underline u}}
\newcommand{\uc}{{\underline c}}

\newcommand\tc{t_{\mathrm{c}}}
\newcommand\tcx{t_{\mathrm{b}}}
\newcommand\tcs{t_{\mathrm{s}}}
\newcommand{\ts}{t} 

\newcommand{\St}{{\tilde{S}}}

\newcommand{\cTv}[1][{v,\ts}]{{\mathcal T}_{{#1}}}
\newcommand{\cVv}[1][{v,\ts}]{{\mathcal V}_{{#1}}}
\newcommand{\tpv}[1][{v,\ts}]{{\mathfrak T}_{{#1}}}
\newcommand{\tpvV}[1][{v,\ts}]{{\mathfrak V}_{{#1}}}
\newcommand\bpt[1][{\ts}]{{\mathfrak T}_{{#1}}}
\newcommand\bpV[1][{\varphi,\ts}]{{\mathfrak V}_{{#1}}}
\newcommand\bp[1][{\varphi,\ts}]{{\mathfrak X}_{{#1}}}
\newcommand\tF{{\widetilde F}}
\newcommand\norm[1]{||#1||}

\begin{document}
\title{The evolution of subcritical Achlioptas processes}
\author{Oliver Riordan and Lutz Warnke%
\thanks{Mathematical Institute, University of Oxford, 24--29 St Giles', Oxford OX1 3LB, UK.
E-mail: {\tt \{riordan,warnke\}@maths.ox.ac.uk}.}}
\date{April 23, 2012}
\maketitle

\begin{abstract}
In Achlioptas processes, starting from an empty graph, in each step two 
potential edges are chosen uniformly at random, and using some rule one of them 
is selected and added to the evolving graph. Although the evolution of such 
`local' modifications of the Erd{\H o}s--R\'enyi random graph process has 
received considerable attention during the last decade, so far only rather 
simple rules are well understood. Indeed, the main focus has been on 
`bounded-size' rules, where all component sizes larger than some constant $B$ 
are treated the same way, and for more complex rules very few rigorous results 
are known. 

In this paper we study Achlioptas processes given by (unbounded) size rules 
such as the sum and product rules. Using a variant of the neighbourhood 
exploration process and branching process arguments we show that certain key 
statistics are tightly concentrated at least until the susceptibility (the 
expected size of the component containing a randomly chosen vertex) diverges. 
Our convergence result is most likely best possible for certain rules: in the 
later evolution the number of vertices in small components may not be 
concentrated. Furthermore, we believe that for a large class of rules the 
critical time where the susceptibility `blows up' coincides with the 
percolation threshold. 
\end{abstract}

\section{Introduction}
In 2000 Dimitris Achlioptas suggested a class of variants of the classical 
random graph process, now called \emph{Achlioptas processes}. Such a process 
defines, for each $n$, a random sequence $(G_i)_{i \ge 0}=(G_i^{\cR})_{i \ge 0}$ 
of graphs with vertex set $[n]$ as follows: start with an empty graph 
$G_{0}$ on $n$ vertices. At each step $i \ge 1$, two potential edges $e_1$ and 
$e_2$ are chosen independently and uniformly at random from all $\binom{n}{2}$ 
possible edges (or from those edges not present in $G_{i-1}$). One of these 
edges is selected according to a rule $\cR$ and added to the graph, so 
$G_{i}=G_{i-1}\cup \{e\}$ for $e=e_1$ or $e_2$. (As usual in combinatorics,
we omit the dependence on $n$ in the notation to avoid clutter.) 
Achlioptas processes are a special case of the more general class of 
\emph{$\ell$-vertex rules}, where in each step $\ell \ge 2$ vertices 
$v_1, \ldots, v_{\ell}$ are chosen independently and uniformly at random and then at least
(usually exactly) one edge between these vertices is added. 
Always adding $e=e_1$ (or $e=v_1v_2$) gives the Erd\H os--R\'enyi random graph 
process, which has been extensively studied for more than 50 years; by now many 
of its properties are extremely well understood, in particular the evolution of 
the component structure, see e.g.~\cite{BBRG,BR-RG,JKLP1993}.

During the last decade the evolution of certain `simple' Achlioptas processes 
has received considerable attention, mainly for so-called \emph{bounded-size} 
rules, see e.g.~\cite{BBW2011,BF2001,BohmanKravitz2006,JansonSpencer2010,KPS2011,AAP2011,SpencerWormald2007}. 
These make their decisions based only on the sizes of the components containing 
the endvertices of $e_1$ and $e_2$, with the restriction that all sizes larger 
than some constant $B$ are treated in the same way. For bounded-size rules a 
number of results have been established concerning concentration of the number 
of vertices in components of fixed size~\cite{BohmanKravitz2006,SpencerWormald2007}, 
the size of the second largest component~\cite{AAP2011}, the existence and 
location of the percolation phase transition where the (unique) linear size 
`giant' component first emerges~\cite{BohmanKravitz2006,SpencerWormald2007}, 
the `order' of the phase transition~\cite{AAP2011}, and the rescaled size of 
the largest component~\cite{AAP2011}. For one particular bounded-size rule (a 
variant of a process suggested by Bohman and Frieze~\cite{BF2001}), some finer 
details of the percolation phase transition have recently been 
investigated~\cite{BBW2011,JansonSpencer2010,KPS2011}. 
The punchline of the work mentioned above is that bounded-size rules seem to 
show many qualitative similarities with the `classical' Erd\H os--R\'enyi 
random graph process; in the language of mathematical physics they appear to 
be in the same `universality class'.

In contrast, for more involved Achlioptas processes very few rigorous results 
are known, although these have been widely studied in recent years. 
This in particular applies to the class of `unbounded'-size rules, usually 
simply called \emph{size rules}, whose choices depend only on the sizes of 
the four components containing the endvertices of the two offered edges. 
To illustrate our very limited understanding of these, we mention that in one 
line of research, stimulated by a conjecture of Achlioptas, D'Souza and 
Spencer published in \emph{Science}~\cite{Science2009} (based on `conclusive 
numerical evidence'), it was believed that certain size rules (in particular 
the \emph{product rule}) could give rise to a discontinuous (`first 
order') phase transition. However, recently it was rigorously shown 
in~\cite{AAP2011,Science2011} that the phase transition is in fact continuous 
for all Achlioptas processes (even for a larger class of processes). The 
surprises that (unbounded) size rules have shown so far indicate that our 
intuition for these processes still needs to be developed. 
In fact, obtaining non-trivial results for involved Achlioptas processes using 
e.g.\ the product rule is well known to be a technical challenge (see 
e.g.~\cite{Science2009,Science2011J}).

Given a graph $G$, let $N_k(G)$ denote 
the number of vertices of $G$ in components of size $k$, and define the 
\emph{susceptibility} of $G$ as $S(G)=\sum k N_k(G)/n$, so $S(G)$ is the 
expected size of the component containing a randomly chosen vertex.
Let $L_1(G)$ denote the number of vertices in the (a, if there is a tie) largest 
component of $G$. We say that $\tc=\tc^{\cR}$ is the \emph{percolation threshold} 
for the process $(G_i^{\cR})$ if for $t<\tc$ whp 
$L_1(G_{tn})=o(n)$ while for $t>\tc$ whp $L_1(G_{tn})=\Omega(n)$
(as usual, we henceforth ignore the irrelevant rounding to integers, and say 
that an event holds \emph{whp}, if it holds with probability tending to $1$ 
as $n\to\infty$).
The analysis for bounded-size rules in~\cite{BohmanKravitz2006,JansonSpencer2010,SpencerWormald2007} 
uses $S(G_i)$ as well as $N_k(G_i)$ for fixed $k$ as key statistics. In the 
`subcritical' regime $t<\tc$, they
use Wormald's differential equation method~\cite{Wormald1995DEM,Wormald1999DEM} 
to establish the existence of functions $\rho_k=\rho_k^{\cR}$, $s=s^{\cR}$ such 
that $N_k(G_{tn}) \approx \rho_k(t)n$ and $S(G_{tn}) \approx s(t)$ hold whp.
Based on this they show that
$\tc=\tc^{\cR}$ is given by the blow-up point of the susceptibility:
$\lim_{t \nearrow \tc} s(t) = \infty$.
As indicated by Spencer and Wormald~\cite{SpencerWormald2007}, 
for general size rules the approximation of the key statistics using the 
differential equation method seems difficult. Another intricacy for such rules 
is that the dependencies among the selected edges often seem to be more complex 
in comparison with bounded-size rules, for which a large subset of the added 
edges can be thought of as chosen uniformly at random (essentially, when all 
four endvertices are in components of size larger than $B$); we return to this 
in Section~\ref{sec:disc}.

\subsection{Main result}\label{sec:main}
In this paper we establish the first rigorous convergence result for Achlioptas 
processes using unbounded size rules such as the product rule: we show that the 
number of vertices in components of size $k \ge 1$ (and the susceptibility) is 
tightly concentrated until the susceptibility `blows up', which happens at a 
critical time $\tcx$. In fact, our result holds for a very large class of 
Achlioptas-like processes, including essentially all Achlioptas processes 
studied so far (see Section~\ref{sec:evo} for the formal definition of $\ell$-vertex 
size rules). Here $S(G_{tn}^{\cR}) \pto \infty$ as $n \to \infty$ means that 
for any $C >0$ we have $\Pr(S(G_{tn}^{\cR}) \le C) \to 0$ as $n \to \infty$, 
and $N_{\ge k}(G)$ denotes the number of vertices of $G$ in components of size 
at least $k$. 
\begin{theorem}\label{thm:main}
Let $\ell \ge 2$ and let $\cR$ be an $\ell$-vertex size rule. There exist 
$\tcx=\tcx^{\cR} \in [\frac{1}{\ell(\ell-1)},1]$ and functions 
$(\rho_k)_{k \ge 1}$ with $\rho_k=\rho_k^{\cR}:[0,\tcx)\to[0,1]$ such that the 
following holds. For every $t \ge \tcx$ we have 
\begin{equation}\label{eq:S2:infty}
S(G_{tn}^{\cR}) \pto \infty
\end{equation}
as $n \to \infty$. For every $t < \tcx$ we have $\sum_{k \ge 1}\rho_k(t)=1$. 
Also, for every $t < \tcx$ there exist $a,A,C>0$ (depending only on 
$\cR,\ell,t$) such that for every $t' \in [0,t]$ we have 
$\rho_{k}(t') \le Ae^{-ak}$ for all $k \ge 1$. In addition, for 
$n \ge n_0(\cR,\ell,t)$ the following holds with probability at least 
$1-n^{-99}$: for every $0 \le i \le tn$ we have 
\begin{gather}
\label{eq:Nk:C:rho}
|N_{k}(G_{i}^{\cR}) - \rho_k(i/n)n| \le (\log n)^C n^{1/2} \quad \text{for all $k \ge 1$,}\\
\label{eq:S:C:rho}
|S(G_{i}^{\cR}) - {\textstyle \sum_{k \ge 1}}k \rho_k(i/n)| \le (\log n)^C n^{-1/2}, 
\end{gather}
and $N_{\ge k}(G_{i}^{\cR}) \le Ae^{-ak}n$ for all $k \ge 1$. 
\end{theorem}
To interpret this result, we think of the functions $\rho_k(t)$ as describing 
the `scaling limit' of the component size distribution at `time' $t<\tcx$, where 
time is the number of steps divided by $n$. A key aspect of the result is 
that this limit does not depend on $n$; in 
fact, most of our technical work is devoted to establishing this property -- 
to show only that $N_k(G_{tn}^{\cR})$ is concentrated around its expectation, 
simpler arguments would suffice. 
The tail bound on $\rho_k$ given in Theorem~\ref{thm:main} 
states that the idealized component size distribution has an exponential tail 
for $t<\tcx$, as one would expect in a strictly sub-critical random graph. 
It implies that 
$s(t)=\sum_{k \ge 1} k \rho_k(t) < \infty$ if $t < \tcx$, so \eqref{eq:S:C:rho} 
implies that that for $t < \tcx$ we have 
\begin{equation}\label{Sbelow}
 S(G_{tn}^{\cR}) \pto s(t)<\infty,
\end{equation}
where $\pto$ denotes convergence in probability. 
The proof of Theorem~\ref{thm:main} will show that $s(\tcx-\eps)\ge (\ell(\ell-1)\eps)^{-1}$, 
so the (idealized) susceptibility $s(t)$ blows up at $\tcx$. The last 
statement of the theorem implies that for $t<\tcx$ we have $L_1(G_{tn}^{\cR})\le B_t\log n$ 
whp, for some constant $B_t$ that depends on $t$. Finally, we shall 
show in the Appendix that (unless $\cR$ directly adds cycles to the graph), 
for $t<\tcx$ whp almost all components are trees, with the rest unicyclic. 

Theorem~\ref{thm:main} allows us to say something about what happens at time $t=\tcx$. 
Indeed, the definition of an $\ell$-vertex rule ensures that in one step, at most $\ell$ 
components are destroyed and at most $\ell$ (in fact at most $\ell/2$) are created, 
so $|N_k(G_{i+1}^{\cR})-N_k(G_i^{\cR})|\le \ell k$. It follows that each $\rho_k$ 
is Lipschitz continuous on $[0,\tcx)$ with constant $k\ell$. Hence we can 
extend each $\rho_k$ continuously to the point $\tcx$, and \eqref{eq:Nk:C:rho} 
and the Lipschitz properties of $N_k$ and $\rho_k$ imply that 
\begin{equation}\label{Nktcx}
 N_k(G_{\tcx n}^{\cR})/n \pto \rho_k(\tcx). 
\end{equation}
Together with Theorem~\ref{thm:main}, the continuity results of 
\cite{AAP2011,Science2011} imply that $L_1(G_{\tcx n})/n\pto 0$, 
and that $\sum_k \rho_k(\tcx)=1$, 
so the numbers $(\rho_k(\tcx))_{k\ge1}$ do capture the asymptotic component size 
distribution of $G_{\tcx n}^{\cR}$, although we do not have such tight 
error bounds as in \eqref{eq:Nk:C:rho}. 

The proof of Theorem~\ref{thm:main} is based on a variant of the neighbourhood 
exploration process and relies on branching process (approximation) arguments. 
This is quite different from previous approaches in this area, which are based 
on the differential equation method; for certain (restricted) classes of rules 
these establish \emph{local convergence}, i.e., that there exist functions 
$\rho_k=\rho_k^{\cR}:\Rp\to [0,1]$ such that, for each fixed $k\ge 1$ and 
$t\ge 0$, we have $N_{k}(G_{tn}^{\cR})/n \pto \rho_k(t)$ as $n\to\infty$.
The limitations of these 
approaches are that they (i) only apply to certain bounded-size 
rules~\cite{BohmanKravitz2006,SpencerWormald2007}, or (ii) when applied to size 
rules need the additional assumption that certain systems of differential 
equations have unique solutions~\cite{RW2011}, which is not known to hold for 
the product rule, for example. So, Theorem~\ref{thm:main} establishes for the 
first time (a strong form of) local convergence for unbounded size rules (such 
as the product rule) until the susceptibility diverges. 
We believe that this convergence result is best possible: on the basis of 
heuristics and simulations presented in~\cite{arxiv} we believe that there 
are certain natural size rules for which beyond $t=\tcx$ a giant component 
emerges whose size is not concentrated. In these rules the numbers of vertices 
in components of each fixed size $k$ are presumably also not concentrated after 
this point.

Theorem~\ref{thm:main} has some analogies with `classical' percolation theory 
on, for example, the infinite lattice $\ZZ^d$, where there are two \emph{a 
priori} different critical probabilities $p_\mathrm{H}$ and $p_\mathrm{T}$. 
Intuitively, these correspond to the thresholds for (i) having (with positive 
probability) an infinite cluster and (ii) the expected cluster size being 
infinite. For essentially all `natural' lattices of interest it is nowadays 
known that $p_\mathrm{H}=p_\mathrm{T}$ (see e.g.~\cite{AB1987,Menshikov1986}), 
but this fact is not at all obvious! Note that in the finite setting of this 
paper these two properties correspond to (i) having a linear size component, 
and (ii) diverging susceptibility. 
More formally, define $\tc=\tc^{\cR}$ as the supremum of the set of $t \ge 0$ 
for which $L_1(G_{tn}^{\cR})/n \pto 0$ as $n \to \infty$, and $\tcx=\tcx^{\cR}$ 
as the supremum of the set of $t \ge 0$ for which $S(G_{tn}^{\cR})$ is bounded in 
probability. Note that $L_1(G) \le \sqrt{n S(G)}$ implies $\tcx \le \tc$.
The remark after 
Theorem~\ref{thm:main} entails that for size rules $\tcx$ is equal to the 
infimum of the set of $t \ge 0$ for which \eqref{eq:S2:infty} holds, and that 
$S(G_{\tcx n}^{\cR}) \pto \infty$. In fact, we believe that both thresholds 
coincide for size rules (analogous to the `classical' case).
\begin{conjecture}\label{conj}
Let $\ell \ge 2$ and let $\cR$ be an $\ell$-vertex size rule. 
Then $\tcx^{\cR}=\tc^{\cR}$. More precisely, for any $t > \tcx^{\cR}$ and 
$\eps > 0$ there exist $\delta,n_0 > 0$ (depending only on $\cR,\ell,t,\eps$)
such that $\Pr(L_1(G_{tn}^{\cR}) \ge \delta n) \ge 1-\eps$ for $n \ge n_0$. 
\end{conjecture}
Recall that Achlioptas processes (where the choice is between two edges) are a 
sub-class of $4$-vertex rules. Conjecture~\ref{conj} was proved for bounded-size 
Achlioptas processes by Spencer and Wormald~\cite{SpencerWormald2007}, 
and for a subset of these processes by Bohman and Kravitz~\cite{BohmanKravitz2006}. 
We shall show in Section~\ref{sec:unif} that it holds for all bounded-size 
$\ell$-vertex rules, as well as many other size rules, including the `reverse 
product rule', for example. However, it does not hold for general $\ell$-vertex 
rules. Indeed, in Section~\ref{sec:delay} we show that modified size rules with 
one additional feature, namely that they may once switch their behaviour based 
on the number $n$ of vertices and the number $i$ of steps (or the value of the 
susceptibility), can delay the appearance of a linear size component for 
$\Omega(n)$ steps beyond the point where the susceptibility diverges.

\section{Evolution of Achlioptas processes with an initial graph}\label{sec:evo} 
In this paper we consider the evolution of Achlioptas processes starting with 
an initial graph $F$ with vertex set $V=[n]$; we restrict our attention 
to \emph{$\ell$-vertex size rules $\cR$}, whose decisions depend only on the 
sizes of the components containing the randomly chosen vertices. 
More precisely, each such rule $\cR$ yields a random sequence 
$(F_{i}^{\cR})_{i\ge 0}$ of graphs on $V$ with $F_{0}^{\cR}=F$. For every 
$i\ge 0$ we draw $\ell$ vertices $\uv_{i+1}=(v_{1},\ldots,v_{\ell})$ from $V$ 
independently and uniformly at random, and then, writing 
$\uc_{i+1}=(c_{1},\ldots,c_{\ell})$ for the sizes of the components containing 
$v_{1},\ldots,v_{\ell}$ in $F_{i}^{\cR}$, we obtain obtain $F_{i+1}^{\cR}$ by 
adding a non-empty set of edges $E_{i+1}$ to $F_{i}^{\cR}$, where $\cR$ 
deterministically selects $E_{i+1}$ as a subset of all pairs between vertices 
in $\uv_{i+1}$ based only on the list of component sizes $\uc_{i+1}$. 
Usually exactly one edge is added, but there is no reason to insist on this.

When $F=G_0$ is the empty graph on $n$ vertices we obtain the `standard' 
Achlioptas processes using $\ell$-vertex size rules $\cR$ as defined 
in~\cite{AAP2011}. 
As usual, we can allow for small variations in the above definition; this 
includes, for example, each time picking an $\ell$-tuple of \emph{distinct} 
vertices, or picking (the ends of) $\ell/2$ randomly selected (distinct) 
edges not already present in $G_{i}^{\cR}$, see also~\cite{AAP2011}. 
For $\ell=2$ we thus recover the `classical' Erd\H os--R\'enyi random graph 
process by always adding the pair $v_1v_2$. 
In addition, our proofs can be written to allow $\cR$ to make randomized 
decisions(with the probability of adding some set of edges depending only on 
$\uc_{i+1}$), and, furthermore, to allow $\cR$ to know which vertices in 
$\uv_{i+1}$ are in the same components of $F_{i}^{\cR}$ (for compatibility 
with~\cite{AAP2011} we then require $E_{i+1}\neq \emptyset$ whenever all 
$v_j$ are in distinct components, although nothing in the proof of 
Theorem~\ref{thm:main}, except for the bound $\tcx \le 1$, needs this).

One difficulty in the proof of Theorem~\ref{thm:main} is that there is a 
complicated dependence between the decisions of $\cR$ in each round (and their 
order is also important). Indeed, changes can `propagate' throughout the 
process: if the sizes of a few components are modified (e.g.\ by altering 
decisions of $\cR$ or tuples $\uv=(v_1, \ldots, v_{\ell})$ offered), then this 
might change many future decisions of $\cR$, which in turn might alter further 
decisions, etc. To overcome this our proof proceeds by induction, always 
establishing concentration only for a small number of steps; this is also the 
reason why we study the more general evolution starting from an initial graph 
$F$. Each time we rely on a two-round exposure argument: in the first round we 
reveal which tuples are selected, and in the second we then expose their 
order. For size rules \emph{not} all tuples and components of $F$ `influence' 
the size of the component in $F_{i}^{\cR}$ containing $v$: only those which can 
be reached from $v$ after adding \emph{all} pairs of each $\ell$-tuple to the 
graph (every rule only selects a subset of these pairs). The key observation is 
now that given the corresponding `relevant' tuples and components of $F$ of the 
first round, the order of these tuples (exposed in the second round) determines 
the size of the component containing $v$. It turns out that if we only consider 
$\sigma n$ rounds for $\sigma$ sufficiently small, then an exploration process 
determining these relevant tuples and components in the first round can be 
closely approximated by a subcritical branching process $\bp[{\sigma}]$ which 
is defined \emph{without} reference to $n$. Since the outcome of the second 
round is a (random) function of the first one, it thus seems plausible that 
$\E N_k(F_{\sigma n}^{\cR})/n$ is independent of $n$ (up to small error 
terms). In addition, since the first round is subcritical, this means that 
there typically are not too many tuples and components which influence the size 
of the component containing $v$. At least on an intuitive level this makes it 
plausible that it should be possible to establish concentration of 
$N_k(F_{i}^{\cR})$ around its expectation by applying McDiarmid's inequality.

The rest of this paper is organized as follows. In the next section we 
state our main technical result (Theorem~\ref{thm:evolution}), and then show in 
Section~\ref{sec:proof} how it implies Theorem~\ref{thm:main}. Afterwards, 
in Section~\ref{sec:preliminaries} we present some branching process 
preliminaries; these are used in Section~\ref{sec:proof:thm:evolution}, where 
we establish Theorem~\ref{thm:evolution}.  In Section~\ref{sec:disc} we  
discuss Conjecture~\ref{conj}, giving examples of classes of size rules 
for which we can prove the conjecture, and examples of non-size 
rules for which it does not hold. Finally, in the appendix we 
consider the cycle structure of Achlioptas processes.

\subsection{Main technical result}\label{sec:evolution}
Our main technical result establishes concentration during the evolution of 
Achlioptas processes starting with an initial graph $F$. The special case of an 
Erd\H os--R\'enyi evolution from an initial graph $F$ (which can be seen as an 
evolving version of a special case of the inhomogeneous random graph model of 
Bollob\'as, Janson and Riordan~\cite{BJR}) has been previously studied by 
Spencer and Wormald~\cite{SpencerWormald2007} and Janson and 
Spencer~\cite{JansonSpencer2010}, the main focus being on the size of the 
largest component. In this context the susceptibility turns out to be the key 
parameter, and both papers use in essential ways that the Erd\H os--R\'enyi 
evolution corresponds to the addition of \emph{uniform} random edges (or 
pairs of vertices). In contrast, when studying the evolution of Achlioptas 
processes, we need to deal with intricate dependencies between the edges added.

Using susceptibility as a guide, we now briefly motivate the number of steps 
our result applies to. Suppose that, starting with $F$ satisfying $S(F) = L$, 
we use the rule $\cI$ which in each step joins all $\ell$ random vertices by 
edges. Set $s(t) = S(F_{tn}^{\cI})$. If the sizes of the joined components are 
$c_i$, then (assuming that all components are distinct) the susceptibility 
changes by $(\sum c_i)^2/n-\sum c_i^2/n=\sum_{i \neq j} c_i c_j/n$. So, since 
the vertices of each tuple are chosen uniformly at random, it seems plausible 
that typicality we have 
$s'(t) \approx n \E(S(F_{tn+1}^{\cI})-S(F_{tn}^{\cI})) \approx \ell(\ell-1) s(t)^2$. 
For $t < [\ell(\ell-1)L]^{-1}= \tcs$ this suggests 
$s(t) \approx [1/L-\ell(\ell-1)t]^{-1}$. Since in each step any rule $\cR$ only 
adds a subset of all $\binom{\ell}{2}$ pairs to the graph, this indicates that 
the susceptibility does not `blow up' as long as $t < \tcs$. The following 
result confirms this heuristic argument and shows that, under suitable 
conditions, for $t < \tcs$ the number of vertices in components of size 
$k \ge 1$ is also tightly concentrated (the function $\varrho$ intuitively 
results from an `infinite' version of the rule $\cR$). Here we set 
$\chi(\varphi) = \sum_{k \ge 1} k \varphi(k)$ and 
$\chi(\varrho,t) = \sum_{k \ge 1} k \varrho(k,t)$, 
and write $x=a \pm b$ as shorthand for $x \in [a-b,a+b]$. 
\begin{theorem}\label{thm:evolution}
Let $\ell \ge 2$ and let $\cR$ be an $\ell$-vertex size rule. 
Suppose $\beta > 1$, $B > 0$, $L \ge 1$ and $\varphi:\NN \to [0,1]$ satisfy
\begin{gather} 
\label{eq:varphi:technical}
{\textstyle \sum_{k \ge 1}} \varphi(k) = 1,\\
\label{eq:varphi:tailsum}
{\textstyle \sum_{k \ge 1}} \varphi(k) \beta^k \le B,\\
\label{eq:varphi:chi}
\chi(\varphi) \le L. 
\end{gather} 
There is a function $\varrho:\NN\times\Rp\to[0,1]$ (depending only on 
$\varphi,\cR,\ell$) such that for all $\sigma \ge 0$ satisfying 
\begin{equation} 
\label{eq:Nk:sigma}
\sigma < [\ell(\ell-1)L]^{-1}
\end{equation} 
there exist $\tilde{\beta},\tilde{B},\tilde{L} > 1$ (depending only on $\ell,L,\sigma,\beta,B$) 
such that for every $t \in [0,\sigma]$ equations 
\eqref{eq:varphi:technical}--\eqref{eq:varphi:chi} hold when 
$\beta,B,L,\varphi(\cdot)$ are replaced by 
$\tilde{\beta},\tilde{B},\tilde{L},\varrho(\cdot,t)$. 
If in addition $F$ is a graph on $n$ vertices which for $C \ge 0$ satisfies 
\begin{gather} 
\label{eq:Nk:approx}
N_{k}(F) = \varphi(k) n \pm (\log n)^C n^{1/2} \quad \text{for all $k \ge 1$},\\
\label{eq:NK:tailsum}
{\textstyle \sum_{k \in [n]}} N_{k}(F) \beta^j \le B n, \\
\label{eq:NK:chi}
S(F) \le L, 
\end{gather} 
then, setting $\tilde{C}=C+9$, for $n \ge n_0(\ell,L,\sigma,\beta,B,C)$ the 
following holds with probability at least $1-n^{-200}$: for every 
$0 \le i \le \sigma n$ we have 
\begin{equation}\label{eq:S2}
S(F_{i}^{\cR}) = \chi(\varrho,i/n) \pm (\log n)^{\tilde{C}} n^{-1/2}, 
\end{equation} 
and equations \eqref{eq:Nk:approx}--\eqref{eq:NK:chi} hold when 
$\beta,B,L,C,F,\varphi(\cdot)$ are replaced by 
$\tilde{\beta},\tilde{B},\tilde{L},\tilde{C},F_{i}^{\cR},\varrho(\cdot,i/n)$. 
\end{theorem}
The proof of Theorem~\ref{thm:evolution} is quite involved and is deferred to 
Section~\ref{sec:proof:thm:evolution}. It is useful to observe that since 
$\beta >1$ holds, \eqref{eq:varphi:tailsum} and \eqref{eq:NK:tailsum} imply the 
tail bounds $\max\{\sum_{j \ge k}\varphi(j),N_{\ge k}(F)/n\} \le B \beta^{-k}$ 
for all $k \ge 1$, so $L_1(F)=O(\log n)$. By Theorem~\ref{thm:evolution} 
analogous estimates also hold for $F_{i}^{\cR}$ with $i \le \sigma n$. In fact, 
for \eqref{eq:NK:tailsum}, \eqref{eq:NK:chi} to hold with $\beta,B,L,F$ replaced 
by $\tilde{\beta},\tilde{B},\tilde{L},F_{i}^{\cR}$, a minor modification of our 
proof shows that it suffices to assume \eqref{eq:Nk:sigma}, 
\eqref{eq:NK:tailsum} and \eqref{eq:NK:chi} only; for the special case $\ell=2$ 
this was established by Spencer and Wormald~\cite{SpencerWormald2007} under 
similar conditions. However, the key point of Theorem~\ref{thm:evolution} is 
\eqref{eq:Nk:approx}, i.e., that we obtain concentration of number of vertices 
in components of size $k$.

Turning to the susceptibility, by combining \eqref{eq:Nk:approx} with the tail 
bounds following from \eqref{eq:varphi:tailsum} and \eqref{eq:NK:tailsum}, for 
each fixed $j$ we readily obtain rather precise estimates for 
$S_j(F_{i}^{\cR})=\sum_{k \in [n]}k^{j}N_{k}(F_{i}^{\cR})/n$ with 
$i \le \sigma n$, similar to \eqref{eq:S2}. Furthermore, since 
$L_1(F_{\sigma n}^{\cR})=O(\log n)$ whp, we can easily use the differential 
equation method to make our heuristic discussion regarding the susceptibility 
rigorous, which e.g.\ yields 
$\chi(\varrho,\sigma) \le [1/L-\ell(\ell-1)\sigma]^{-1}$ (for the special case 
$\ell=2$ this was noted by Bohman et.~al~\cite{BFKLS2011}; it is also implicit 
in~\cite{SpencerWormald2007}). However, this crude bound, which follows from 
always connecting all $\ell$ vertices by edges in each step, is generally far 
from the truth; for this reason it does not suffice in our inductive 
application of Theorem~\ref{thm:evolution}, where we use the `correct' value 
given by \eqref{eq:S2}.

\subsection{Proof of Theorem~\ref{thm:main}}\label{sec:proof}
This section is devoted to the proof of Theorem~\ref{thm:main}, which we 
establish by an inductive application of Theorem~\ref{thm:evolution}: each time 
we show concentration during a small number of steps (and maintain certain 
technical conditions), where the lengths of these intervals decrease as the 
susceptibility increases. This is also the main idea of the following rather 
technical construction: as we shall see in the proof of Lemma~\ref{lem:main}, 
for each interval of length $\Delta_j$ it determines the scaling limits 
$\rho_k$ (and certain tail bounds) in a way that does \emph{not} depend on $n$.

We inductively define a sequence $(\beta_j,B_j,\varrho_j,\Delta_j,L_j)_{j \ge 0}$ 
with $\beta_j > 1$, $B_j>0$, $\Delta_j \ge 0$, $L_j \ge 1$ and 
$\varrho_j:\NN\times\Rp \to [0,1]$, where the $\beta_j$ are decreasing 
($\beta_{j+1} \le \beta_j$) and the $B_j$ are increasing ($B_{j+1} \ge B_j$). 
In addition, for each $j \ge 0$ the sequence satisfies the following 
invariant: for every $t \in [0,\Delta_j]$ equations 
\eqref{eq:varphi:technical} and \eqref{eq:varphi:tailsum} hold with 
$\beta,B,\varphi(\cdot)$ replaced by $\beta_j,B_j,\varrho_j(\cdot,t)$. 
We start by setting $\beta_0=B_0=L_0=2$, $\Delta_0=0$ and defining 
$\varrho_0:\NN\times\Rp \to [0,1]$ with $\varrho_0(1,0)=1$ and 
$\varrho_0(k,t)=0$ otherwise. Given $j \ge 1$, recall that 
$\chi(\varrho_{j-1},t) = \sum_{k \ge 1} k \varrho_{j-1}(k,t)$ and set 
\begin{equation}\label{eq:tjp1}
L_j=\chi(\varrho_{j-1},\Delta_{j-1})+1 \quad \text{and} \quad \Delta_{j}=[\ell(\ell-1)\left(L_j+1\right)]^{-1} .
\end{equation}
Applying the first part of Theorem~\ref{thm:evolution} with $\beta=\beta_{j-1}$, 
$B=B_{j-1}$, $L=L_j$ and $\varphi(k)=\varrho_{j-1}(k,\Delta_{j-1})$, we use the 
resulting $\varrho$ to define $\varrho_{j}=\varrho$. Furthermore, by 
considering $\sigma=\Delta_{j}$ we obtain $\tilde{\beta},\tilde{B}$ and set 
$\beta_{j}=\min\{\tilde{\beta},\beta_{j-1}\}$ and 
$B_{j}=\max\{\tilde{B},B_{j-1}\}$; by Theorem~\ref{thm:evolution} these satisfy 
the required invariant. Furthermore, it is not difficult to see that the entire 
sequence $(\beta_j,B_j,\varrho_j,\Delta_j,L_j)_{j \ge 0}$ depends only on 
$\cR,\ell$.

Next we combine the $\varrho_j$ (each valid on an interval of length 
$\Delta_j$) to form $\varphi(k,t)$, which will eventually be $\rho_k(t)$ in 
Theorem~\ref{thm:main}; this notation avoids confusion with the $\varrho_j$ 
used. For $t \ge 0$ we define $r_t$ as the smallest $r$ such that 
$t \le \sum_{0 \le j \le r} \Delta_j$ and set $r=\infty$ if no such $r$ 
exists. For all $(k,t) \in \NN \times\Rp$ set 
\begin{equation}\label{def:varphi}
\varphi(k,t)= \begin{cases}
		\varrho_{r_t}(k,t-\sum_{0 \le j < r_t} \Delta_j), & ~~\text{if $r_t < \infty$},\\
		0, & ~~\text{otherwise}.
	\end{cases}
\end{equation}
Transferring this definition to the invariant of the sequence introduced above, 
for all $t \ge 0$ with $r_t < \infty$ it follows that 
\begin{equation} 
\label{eq:Nk:varphi:ind:technical}
{\textstyle\sum_{k \ge 1}} \varphi(k,t) = 1 ,
\end{equation} 
and that for every $t' \in [0,t]$ we have 
\begin{equation} 
\label{eq:Nk:varphi:ind:tailsum}
{\textstyle\sum_{k \ge 1}} \varphi(k,t') \beta_{r_t}^k \le B_{r_t}. 
\end{equation}

Now we are ready to prove the following concentration result, which also 
implies that in the previous construction we always have $r_t < \infty$ if 
\eqref{eq:S2:infty} fails. 
\begin{lemma}\label{lem:main}
Let $\ell \ge 2$ and let $\cR$ be an $\ell$-vertex size rule. For every 
$t \ge 0$ for which \eqref{eq:S2:infty} fails we have $r_t < \infty$, and there 
exist $a,A,C>0$ (depending only on $\cR,\ell,t$) such that the following holds 
for $n \ge n_0(\cR,\ell,t)$ with probability at least $1-n^{-99}$: for every 
$0 \le i \le tn$ we have 
\begin{gather*}
\label{eq:Nk:C:rho:lem}
N_{k}(G_{i}^{\cR}) = \varphi(k,i/n)n \pm (\log n)^{C}n^{1/2} \quad \text{for all $k \ge 1$}, \\
\label{eq:S:C:rho:lem}
S(G_{i}^{\cR}) = {\textstyle\sum_{k \ge 1}}k \varphi(k,i/n) \pm (\log n)^{C}n^{-1/2}, 
\end{gather*}
and $N_{\ge k}(G_{i}^{\cR}) \le Ae^{-ak}n$ for all $k \ge 1$. 
\end{lemma}
\begin{proof}
Given $t \ge 0$, if \eqref{eq:S2:infty} fails there exists $\eps>0$ and an 
infinite subsequence $\bar{n}$ of $\NN$ (depending only on $\cR,\ell,t$) 
satisfying 
\begin{equation} 
\label{eq:S2:bound}
\Pr(S(G_{t \bar{n}}^{\cR}) \le \eps^{-1}) \ge \eps .
\end{equation}
Let $\overline{L}=\eps^{-1}+3$ and $K=\ceil{t\ell(\ell-1)\overline{L}}+1$. Let 
$t_0=0$, and for $j \ge 1$ let 
\begin{equation*}\label{def:t}
t_{j}= \begin{cases}
		t_{j-1}, & ~~\text{if $t_{j-1} > t$},\\
		t_{j-1}+\Delta_j, & ~~\text{otherwise}.
	\end{cases}
\end{equation*}
For $n \ge n_0(\cR,\ell,t)$ we inductively show that for every $0 \le j \le K$, 
setting $C_j=9j+2$, with probability at least $1-j n^{-200}$, for every 
$0 \le i \le t_j n$ we have 
\begin{gather} 
\label{eq:Nk:approx:ind}
N_{k}(G_i^{\cR}) = \varphi(k,i/n) n \pm (\log n)^{C_j} n^{1/2} \quad \text{for all $k \ge 1$,}\\
\label{eq:tailsum:ind}
{\textstyle\sum_{k \in [n]}} N_{k}(G_i^{\cR}) \beta_j^k \le B_j n ,\\
\label{eq:S2:ind}
S(G_{i}^{\cR}) = \chi(\varphi,i/n) \pm (\log n)^{C_j} n^{-1/2} ,
\end{gather}
and for every $0 \le s \le j$ with $t_s \le t$ we have 
\begin{gather} 
\label{eq:chi:ind}
\chi(\varphi,t_s) = \chi(\varrho_s,\Delta_s) < \overline{L}-2 . 
\end{gather}
Note that if $t_K < t$, then substituting \eqref{eq:chi:ind} into 
\eqref{eq:tjp1} yields $\Delta_{j}>[\ell(\ell-1)\overline{L}]^{-1}$ for all 
$1 \le j \le K$. From $K > t\ell(\ell-1)\overline{L}$ it thus follows that 
$t_K > t$, a contradiction. Thus \eqref{eq:chi:ind} implies $t_K \ge t$, i.e., 
$r_t < \infty$. Recall that $(\beta_{j},B_{j})_{j \ge 1}$ and $K$ depend only 
on $\cR,\ell$ and on $\cR,\ell,t$ respectively. Hence the induction hypothesis 
for $j=K$ implies Lemma~\ref{lem:main}, where the tail bounds follow from 
\eqref{eq:tailsum:ind} as $\beta_K > 1$.

For the base case $j=0$ we start with an empty graph on $n$ vertices, and it is 
easy to see that \eqref{eq:Nk:approx:ind}--\eqref{eq:chi:ind} hold with 
$\beta_0=B_0=C_0=2$ and $\varphi(k,0)=\varrho_0(k,0)$, as defined above 
\eqref{eq:tjp1}.

Given $j \ge 1$, for the induction step we may assume that $t_{j-1} \le t$ 
(otherwise $t_{j}=t_{j-1}$, and there is nothing to prove). We first assume 
that $G_{t_{j-1} n}^{\cR}$ satisfies the induction hypothesis, i.e., 
\eqref{eq:Nk:approx:ind}--\eqref{eq:chi:ind} with $j$ replaced by $j-1$. 
In particular, \eqref{eq:Nk:varphi:ind:technical} and 
\eqref{eq:Nk:varphi:ind:tailsum} hold for $t=t_{j-1}$ with $r_t=j-1$, 
and we have 
$S(G_{t_{j-1}n}^{\cR}) \le \chi(\varphi,t_{j-1})+1=\chi(\varrho_{j-1},\Delta_{j-1})+1=L_j$ 
for $n \ge n_0(C_{j-1})$. 
Now we condition on $G_{t_{j-1} n}^{\cR}=F$ and, analogous as after 
\eqref{eq:tjp1}, apply Theorem~\ref{thm:evolution} with $\beta=\beta_{j-1}$, 
$B=B_{j-1}$, $L=L_j$, $\sigma=\Delta_j$, $C=C_{j-1}$ and 
$\varphi(k)=\varrho_{j-1}(k,\Delta_{j-1})=\varphi(k,t_{j-1})$, which is 
possible by the induction hypothesis (and the properties established above). 
So, for $n \ge n_0(\ell,L_j,\Delta_j,\beta_{j-1},B_{j-1},C_{j-1})$, with 
probability at least $1-n^{-200}$, for every $0 \le i \le \Delta_j n$ the graph 
$F_{i}^{\cR}$ satisfies \eqref{eq:varphi:technical}--\eqref{eq:varphi:tailsum}, 
\eqref{eq:Nk:approx}--\eqref{eq:NK:tailsum} and \eqref{eq:S2} when 
$\beta,B,C,F,\varphi(\cdot)$ are replaced by 
$\tilde{\beta},\tilde{B},\tilde{C},F_{i}^{\cR},\varrho(\cdot,i/n)$, where 
$\tilde{C}=C_{j-1}+9$. 
Note that for size rules $F_{\Delta_j n}^{\cR}$ is exactly $G_{t_{j} n}^{\cR}$ 
conditional on $G_{t_{j-1} n}^{\cR}=F$. It is crucial that 
$\tilde{\beta},\tilde{B},\tilde{C},\varrho$ do \emph{not} depend on the 
initial graph $F$, and that by construction $\varrho=\varrho_j$, 
$\beta_j \le \tilde{\beta}$, $B_j \ge \tilde{B}$ and $C_j=\tilde{C}$. 
So, by appealing to the induction hypothesis and recalling \eqref{def:varphi}, 
it follows that with probability at least $1-(j-1)n^{-200}-n^{-200}$ equations 
\eqref{eq:Nk:approx:ind}--\eqref{eq:S2:ind} hold. 
It remains to show that \eqref{eq:chi:ind} holds. To this end 
recall that \eqref{eq:S2:ind} holds with probability at least 
$1-jn^{-200} > 1-\eps$ for all $n \ge n_0({C_{j}},K,\eps)$. So, using that 
susceptibility is monotone increasing, by \eqref{eq:S2:bound} it follows that 
for all $0 \le t' \le \min\{t,t_j\}$ we have 
$\chi(\varphi,t') < \eps^{-1}+1=\overline{L}-2$, say. Now \eqref{eq:chi:ind} 
follows by combining the previous estimate with the observation that for every 
$s \le j$ with $t_s \le t$ we have $\chi(\varphi,t_s) = \chi(\varrho_s,\Delta_s)$. 
This completes the induction step.

Finally, to see that $n \ge n_0(\cR,\ell,t)$ suffices note that in each of the 
$K$ steps we only used 
$n \ge n_0(\ell,L_j,\Delta_j,\beta_{j-1},B_{j-1},C_{j-1},C_{j},K,\eps)$, where 
$C_j=9j+2$ and $L_j,\Delta_j,\beta_{j-1},B_{j-1}$ depend only on $\cR,\ell$. 
This concludes the proof since $\eps$ (and thus $K$) only depends on 
$\cR,\ell,t$. 
\end{proof}
Now we define $\tcx=\tcx^{\cR}$ as the infimum of the set of $t \ge 0$ for 
which \eqref{eq:S2:infty} holds as $n \to \infty$; so \eqref{eq:S2:infty} fails 
for $t < \tcx$. 
The remark after the proof of Lemma~4 in~\cite{AAP2011} implies that for 
$t>1$ we whp have $L_1(G_{tn}^{\cR}) \ge c n$ for $c=c(\ell,t)>0$, yielding 
$S(G_{tn}^{\cR}) \ge [L_1(G_{tn}^{\cR})]^2/n \ge c^2n$; so $\tcx \le 1$. 
Furthermore, for $t<[\ell(\ell-1)]^{-1}$ an application of 
Theorem~\ref{thm:evolution} to the empty graph $F=G_{0}^{\cR}$ on $n$ vertices 
with $\sigma =t$ and $L=1$ (similar as in the proof of Lemma~\ref{lem:main}) 
readily shows $S(G_{\sigma n}^{\cR}) \le \tilde{L}$ whp, so
$\tcx \ge [\ell(\ell-1)]^{-1}$. 
Now suppose that \eqref{eq:S2:infty} fails for $t = \tcx$. The proof of 
Lemma~\ref{lem:main} then shows that whp \eqref{eq:Nk:approx:ind}--\eqref{eq:S2:ind} 
hold for $i=\tcx n$, and that $\chi(\varphi,\tcx) < \overline{L}-2$. 
It follows that we can apply Theorem~\ref{thm:evolution} with 
$\sigma = [\ell^2\overline{L}]^{-1}$ and $L=\overline{L}$; this implies 
$S(G_{(\tcx+\sigma)n}^{\cR}) \le \tilde{L}$ whp, contradicting the definition of 
$\tcx$. So, since the susceptibility is monotone increasing, it follows that 
\eqref{eq:S2:infty} holds for all $t \ge \tcx$. 
Combining our findings, Lemma~\ref{lem:main}, \eqref{eq:Nk:varphi:ind:technical} 
and \eqref{eq:Nk:varphi:ind:tailsum} now yield Theorem~\ref{thm:main} with 
$\rho_k(t)=\varphi(k,t)$.

\subsection{Branching processes preliminaries}\label{sec:preliminaries}
The following basic results for branching processes will be used in the proof of 
Theorem~\ref{thm:evolution}. They are similar to Theorems~3.2 and~3.3 
in~\cite{SpencerWormald2007}, where they are attributed to much earlier results 
of Cr\'{a}mer. Given a non-negative integer valued random variable $X$, let 
$F_X(z) = \E z^X$ denote the \emph{(probability) generating function} of $X$. 
Note that $F_X(z)$ is convex and monotone increasing for $z \ge 0$.

The first lemma essentially states that a two-generation branching process has 
(uniform) exponential tails provided that the generating function of each 
offspring distribution has radius of convergence strictly larger than one 
(and thus also exhibits exponential decay). 
\begin{lemma}\label{lem:bp:twogeneration}
Let $X,Y \ge 0$ be integer valued random variables with $F_X(\alpha) \le A$ 
and $F_Y(\beta) \le B$, where $\alpha, \beta > 1$. 
Let $Z$ be the number of grandchildren in the two-generation branching process 
in which the root node has X children and then each child, independently, has 
$Y$ children. 
There exists $a > 0 $ (depending only on $\alpha,\beta,A,B$) such that 
$\Pr(Z \ge s) \le A e^{-as}$ for all $s \ge 0$. 
\end{lemma}
\begin{proof}
Pick $C \ge \max\{B,2\}$ such that $x=1+(\alpha-1)(\beta-1)/(C-1) \le \beta$. 
Using $F_Y(1) = 1$ and $F_Y(\beta) \le B \le C$, convexity yields 
$F_Y(z) \le [(z-1)C+(\beta-z)]/(\beta-1)$ for all $z \in [1,\beta]$.
So, by choice of $x$ we have $F_Y(x) \le \alpha$. 
Observing that $F_Z(z) = F_X(F_Y(z))$, using monotonicity we obtain 
$F_Z(x) = F_X(F_Y(x)) \le F_X(\alpha) \le A$. 
Since $x > 1$ implies $F_Z(x) \ge \Pr(Z \ge s) x^s$ for every $s \ge 0$, 
we deduce $\Pr(Z \ge s) \le A e^{-as}$ for $a=\log x > 0$, 
completing the proof. 
\end{proof}
The second lemma is a standard result for subcritical Galton--Watson branching 
process: these exhibit (uniform) exponential decay if the offspring 
distribution itself has (uniform) exponential tails. 
\begin{lemma}\label{lem:bp:tree:Z}
Let $Z \ge 0$ be an integer valued random variable with $\E Z \le \mu < 1$ and 
$F_Z(\beta) \le B$, where $\beta>1$. 
Let $T$ be the total size of the Galton--Watson branching process in which each 
node, independently, has $Z$ children. 
There exist $\delta > 1$ and $D>0$ (depending only on $\beta,B,\mu$) such that 
$F_T(\delta) \le D$. 
\end{lemma}
\begin{proof}
Let $f(t) = \E(e^{t(Z-1)})$. Observe that $f(0) = 1$ and 
$f'(0) = \E(Z-1)\le\mu-1$. As in the proof of Lemma~\ref{lem:bp:twogeneration}, 
$F_Z(\beta) \le B$ for $\beta > 1$ yields $\Pr(Z \ge s) \le B \beta^{-s}$, 
which in turn readily implies that for some $C=C(\beta,B)$ we have 
$f''(t) = \E( (Z-1)^2e^{t(Z-1)}) \le C$ for all $0 \le t \le (\log \beta) /2$, 
say. So, using Taylor's theorem, for $0 \le t \le (\log \beta)/2$ we deduce 
$f(t) \le 1 + (\mu-1)t + Ct^2/2 = h(t)$. Let 
$x = \min\{(\log \beta)/2,(1-\mu)/C\} > 0$, and observe that 
$c=\max\{h(x),1/2\} > 0$ satisfies $f(x) \le c < 1$. 
Exploring the branching process tree as usual in breadth-first search order, 
we see that $T > s$ implies $\sum_{i=1}^{s} Z_i \ge s$, where the $Z_i$ are 
independent copies of $Z$ (corresponding to the number of children of the 
$i$-th node). Now, using Markov's inequality and independence of the $Z_i$, for 
every $s \ge 0$ we obtain 
$\Pr(T > s) \le \E(e^{x(\sum_{i=1}^{s} Z_i)}) e^{-xs} = f(x)^s \le c^s$. 
Finally, picking $1 < \delta < 1/c$, it follows that 
$F_T(\delta) \le D=D(\delta,c)$, as claimed. 
\end{proof}

\subsection{Proof of Theorem~\ref{thm:evolution}}\label{sec:proof:thm:evolution}
The proof of Theorem~\ref{thm:evolution} relies on a two-round exposure 
argument: we first reveal the random tuples selected, and afterwards expose 
their order of appearance. It will be convenient to work with a continuous-time 
random graph model, where the $n^{\ell}$ tuples arrive according to independent 
Poisson processes with rates $1/n^{\ell-1}$. So tuples appear with rate $n$, 
and each tuple is chosen uniformly at random and independently of all previous 
choices. Let $E_{t}$ denote the set of tuples which arrive in $[0,t]$; so 
$|E_{t}| \sim \Po(tn)$. Observe that for each tuple $\uu \in [n]^\ell$ the 
number $A_{\uu}(t)$ of its arrivals in $[0,t]$ satisfies 
$A_{\uu}(t) \sim \Po(t/n^{\ell-1})$, and that these random variables are 
independent for different tuples. Furthermore, writing $x=t/n^{\ell-1}$ and 
using $e^{-x} \ge 1-x$ twice, note that for $\ell \ge 2$ we have 
\begin{equation}\label{eq:Auv2}
\Pr(A_{\uu}(t) \ge 2) = 1 - e^{-x}- xe^{-x} \le x(1-e^{-x}) \le x^2 \le t^2 /n^{\ell}.
\end{equation}
Similarly 
\begin{equation}\label{eq:Auv1}
\Pr(A_{\uu}(t) \ge 1) = 1 - e^{-x} \le x = t / n^{\ell-1} .
\end{equation}
Starting with $F$, for each tuple $\uu \in E_{t}$ we join all 
$\binom{\ell}{2}$ pairs of vertices by edges, and we denote the resulting graph 
by $H_t$. We define $H^{\cR}_t$ as the graph which we obtain by starting with 
$F$, and then presenting the tuples to $\cR$ (together with the component sizes 
of the vertices) in a random order, always updating the graph according to the 
decisions of $\cR$ (adding the pairs selected by $\cR$). Since conditioned on 
$|E_{t}|=i$ we have $i$ tuples chosen independently and uniformly at random, 
it follows that 
\begin{equation}
\label{eq:Nk:transfer:Gi}
\E( N_k(H_t^{\cR}) \mid \text{$|E_{t}|=i$}) = \E N_k(F_i^{\cR}).
\end{equation}
Furthermore, mimicking the proof of \emph{Pittel's inequality} (see e.g.~\cite{BBRG}) 
for $0 < tn < n^{\ell}$, a short calculation shows that for any graph property 
$\cQ$ we have 
\begin{equation}\label{eq:pittel}
\Pr(F_{tn}^{\cR} \not\in \cQ) \le 3 \sqrt{t n} \cdot \Pr(H_t^{\cR} \not\in \cQ) .
\end{equation}

In the following sections we always tacitly assume that the assumptions of 
Theorem~\ref{thm:evolution} hold and consider $\ts=\ts(n)$ satisfying 
\begin{equation}\label{eq:ts} 
0 \le \ts \le \sigma \le 1 ,
\end{equation}
where $\sigma \le 1$ follows from \eqref{eq:Nk:sigma}. Furthermore, unless 
stated otherwise, we will use the continuous-time random graph models $H_{\ts}$ 
and $H_{\ts}^{\cR}$. For later usage let 
\begin{equation}\label{eq:U} 
U = (\log n)^{6/5} .
\end{equation}

\subsubsection{Component exploration process for $\ell=2$}\label{sec:NEP}
Our main ingredient for analyzing the first exposure round is a certain 
exploration process. Given a (random) vertex $v$, it finds all tuples in 
$E_{\ts}$ and components of $F$ that are `relevant' in the second exposure 
round for determining $|C_v(H_{\ts}^{\cR})|$, where we write $C_v(G)$ for the 
set of vertices of $G$ that are in the same component as $v$. 
As certain details are rather technical for Achlioptas processes, here we 
first outline some of the basic ideas and techniques for the simpler case of 
an Erd{\H o}s--R\'enyi evolution starting from an initial graph $F$ (in this 
special case similar ideas were used by Spencer and 
Wormald~\cite{SpencerWormald2007}). This formally corresponds to the special 
case $\ell=2$ and the rule which always adds the offered pair $v_1v_2$ to the 
evolving graph; so $H_{\ts}=H_{\ts}^\cR$.

One major difference to the Erd{\H o}s--R\'enyi case (where we start with an 
empty graph on $n$ vertices) is that here we have two sources of edges: (i) the 
initial graph $F$ and (ii) the random pairs in $E_{\ts}$. As edges of type (i) 
are deterministic and those of type (ii) are random, our exploration process 
explicitly considers them separately. 
In the first round we start with a randomly chosen $v$ and mark all 
$u \in C_{v}(F)$ as \emph{reached}; all other vertices are \emph{unreached}. 
In each later round we sequentially go through the vertices $w$ reached in 
the previous round (the order does not matter here) and determine all its 
so far unreached neighbours $u$ in $E_{\ts}$ (corresponding to 
pairs $(u_1, u_2)\in E_{\ts}$ containing $u$ and $w$), each time marking all 
$\tilde{u} \in C_{u}(F)$ as reached. 
Note that upon termination $C_v(H_{\ts})$ equals the set of all reached 
vertices.

The previous procedure yields an associated `exploration tree' 
$\cT_{v}(H_{\ts})$ in a rather natural way: loosely speaking, $u$ is a child of 
$w$ if $u$ was `reached' via $w$. With an eye to the upcoming analysis for size 
rules, here we already introduce different types of nodes: \emph{vertex nodes}, 
\emph{component nodes}, and \emph{root nodes}. More precisely, we define 
$\cT_{v}(H_{\ts})$ inductively as follows: 
it has a root node $v$, whose children are vertex nodes $u \in C_{v}(F)$. Then, 
given any vertex node $w$, each of its so far unreached neighbours $u$ in 
$E_{\ts}$ yields a component node as a child, which in turn has vertex nodes 
$\tilde{u} \in C_{u}(F)$ as children. It follows that the set of all vertex 
nodes in $\cT_{v}(H_{\ts})$ equals $C_v(H_{\ts})$. The main point is that, even 
after ignoring all labels, the \emph{structure} of $\cT_{v}(H_{\ts})$ is enough 
to determine $|C_v(H_{\ts})|$.

The key idea is now to approximate $\cT_{v}(H_{\ts})$ by an `idealized' 
branching process, similar as in the `classical' Erd{\H o}s--R\'enyi case 
(exploiting, as usual, that by construction every edge is tested at most once). 
Recall that in $\cT_{v}(H_{\ts})$ already reached vertices are `ignored'. So, 
noting that endpoints of random pairs in $E_{\ts}$ correspond to random 
vertices, and that each edge gives rise to two \emph{ordered} tuples, it seems 
plausible that $\cT_{v}(H_{\ts})$ is dominated by (may be regarded as a subset 
of) a branching process $\bpt[v,\ts]$ where (ignoring for simplicity the root 
and all labels) every vertex node, independently, has $\Po(2\ts)$ component 
nodes as children, each which in turn, independently, has $N$ vertex node 
descendants, where $N \sim |C_u(F)|$ for a randomly chosen vertex $u$. 
Now, using \eqref{eq:Nk:sigma} and \eqref{eq:NK:chi} each vertex node has 
in expectation $2 \ts \cdot S(F) \le 2 \sigma L < 1$ vertex nodes as 
grandchildren, so we expect that $\bpt[v,\ts]$ resembles a subcritical 
branching process which has $O(\log n)$ size with very high probability. From 
this it follows that $\cT_v(H_{\ts})$ and $\bpt[v,\ts]$ are both small and have 
similar offspring distribution (as not too many vertices are reached and thus 
ignored), so it seems plausible that we can couple them so that they agree whp. 
Note that $\bpt[v,\ts]$ still depends on $n$ and the initial graph $F$. 
Define $\Pr(R=k)=\varphi(k)$, where $\varphi$ is given by 
Theorem~\ref{thm:evolution}. The point is now that using \eqref{eq:Nk:approx} 
it follows that $R$ is very close to $N$. So, denoting by $\bp[\varphi,\ts]$ 
the `idealized' version of $\bpt[v,\ts]$ where we use $R$ instead of $N$, the 
former considerations suggest that there is a coupling such that whp 
$\bpt[v,\ts]\cong\bp[\varphi,\ts]$ holds (ignoring the labels of the 
vertices). To summarize, we just outlined that using the `intermediate' process 
$\bpt[v,\ts]$ we can couple $\cT_v(H_{\ts})$ and $\bp[\varphi,\ts]$ so that 
they typically agree up to isomorphisms. Consequently, the distribution of 
$|C_v(H_{\ts})|$ can be approximated using $\bp[\varphi,\ts]$, which does 
\emph{not} depend on $n$ or $F$.

In the above construction and analysis we used in essential ways that in each 
round only one pair of vertices is chosen and connected by an edge. 
In contrast, when considering Achlioptas processes several vertices 
$\uv=(v_{1}, \ldots, v_{\ell})$ are chosen in each round, and only a 
subset of the edges between these vertices is added to the evolving graph. 
Furthermore, in the second exposure round the order in which the tuples $\uv$ 
are presented matters (as well as the order of the vertices in each tuple). 
This motivates the more involved exploration processes used in the next 
section, whose associated exploration tree captures more detailed 
structural information (also using more types of nodes).

\subsubsection{Component exploration process (the general case)}\label{sec:CEP}
In this section we consider the first exposure round, where the selected set of 
tuples $E_{\ts}$ is revealed. Note that this defines $H_{\ts}$, which we 
obtain by starting with $F$ and then joining all $\ell$ vertices of each tuple 
in $E_{\ts}$ by edges. Using a natural variant of the standard neighbourhood 
exploration process, for any vertex $v$ we can determine $C_{v}(H_{\ts})$ as 
follows. First we determine $C_{v}(F)$, i.e., find all other vertices which are 
in the same component of $F$ as $v$. Then, for each $w \in C_{v}(F)$ we find 
all tuples $\uu=(u_1, \ldots, u_{\ell})\in E_{\ts}$ containing $w$, and repeat 
the same procedure (recursively) for each $u_j \neq w$, see 
Figure~\ref{fig:nexpl}. Observe that for determining $C_{v}(H_{\ts})$ it 
suffices to consider only those vertices $u_j \neq w$ which we have not already 
reached in some previous exploration step.

\begin{figure}[t]
\centering
  \setlength{\unitlength}{1bp}%
  \begin{picture}(280.0, 89.00)(0,0)
  \put(0,0){\includegraphics{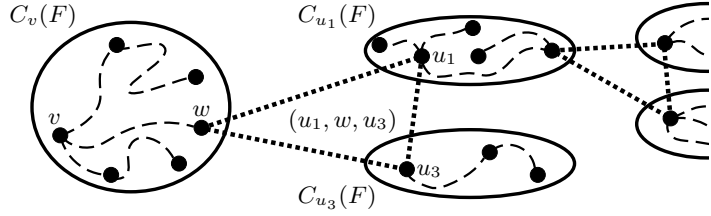}}
  \put(76.16,40.16){\fontsize{8.54}{10.24}\selectfont \makebox[0pt]{$w$}}
  \put(20.79,37.32){\fontsize{8.54}{10.24}\selectfont \makebox[0pt]{$v$}}
  \put(16.53,75.66){\fontsize{8.54}{10.24}\selectfont \makebox[0pt]{$C_v(F)$}}
  \put(167.74,60.75){\fontsize{8.54}{10.24}\selectfont \makebox[0pt]{$u_1$}}
  \put(162.07,18.15){\fontsize{8.54}{10.24}\selectfont \makebox[0pt]{$u_3$}}
  \put(129.41,36.19){\fontsize{8.54}{10.24}\selectfont \makebox[0pt]{$(u_1,w,u_3)$}}
  \put(127.28,75.66){\fontsize{8.54}{10.24}\selectfont \makebox[0pt]{$C_{u_1}(F)$}}
  \put(127.28,7.51){\fontsize{8.54}{10.24}\selectfont \makebox[0pt]{$C_{u_3}(F)$}}
  \end{picture}%
	\caption{\label{fig:nexpl} Example of the neighbourhood exploration process 
for $\ell=3$. It determines $C_{v}(H_{\ts})$ by first finding other vertices 
in the same component of $F$, then finding tuples containing them; afterwards 
it repeats this procedure for the new vertices in those tuples, and so on.}
\end{figure}

In the analysis it is easier to start with a \emph{random} vertex $v$ and break 
down the above exploration process into small steps, constructing an an 
associated \emph{exploration tree} $\cTv=\cTv(F)$. As we shall see, one key 
property of $\cTv$ is that we can (typically) reconstruct the vertices and 
components which have been reached, as well as the tuples which have been 
`tested' so far. The vertices of each exploration tree have different types: 
\emph{vertex nodes}, \emph{component nodes} and \emph{tuple nodes} will 
represent vertices, components of $F$ and $\ell$-tuples, respectively. For 
technical reasons we also have \emph{root nodes} and \emph{index nodes}. 
We denote the vertex nodes of $\cTv$ by $\cVv$.

As mentioned above, our exploration starts with a random vertex $v$, which 
serves as the root node of $\cTv$, see Figure~\ref{fig:Tv}. Next we 
(deterministically) find all vertices $w \in C_{v}(F)$ and then add the vertex 
nodes $w$ as children of the root. In the following we sequentially traverse 
each level containing vertex nodes (which essentially corresponds to a breadth 
first search). Given a vertex node $w$, we add $\ell$ index nodes 
$w_1, \ldots, w_{\ell}$ as children, where $w_j$ is an index node of type $j$. 
For each $j=1, \ldots, \ell$ we sequentially test for the presence and 
multiplicity of all so far untested tuples $\uu =(u_1, \ldots, u_\ell)$ 
with $u_j=w$; we denote the resulting multiset of found tuples by $S_{j,w}$. 
Now we sequentially traverse the $\uu \in S_{j,w}$. For each such 
$\uu =(u_1, \ldots, u_\ell)$ we add a tuple node $\uu $ and 
traverse the $u_i$ with $i \neq j$ sequentially. For each $i \neq j$, we add 
a component node $u_i$ of type $\lambda_j(i)$ as a child of $\uu $, 
where $\lambda_j(i)=i$ for $i<j$ and $\lambda_j(i)=i-1$ for $i > j$ (so that 
the component nodes $\{u_i\}$ with $i \neq j$ have types $1, \ldots, \ell-1$). 
If $u_i$ is already contained in $\cTv$ then we `ignore' this component node. 
Otherwise we add vertex nodes $w \in C_{u_i}(F)$ as children of $u_i$, see 
Figure~\ref{fig:Tv}. Note that $C_v(H_{\ts})$ consists exactly of the union of 
all vertex nodes of $\cTv$, so 
\begin{equation}\label{eq:Vv}
C_v(H_{\ts}^{\cR}) \subseteq C_v(H_{\ts}) = \cVv .
\end{equation}

The main point is that whenever no component nodes are ignored, then from 
$\cTv$ we can reconstruct all explored tuples (in $E_{\ts}$) and components 
(of $F$), which for size rules are the only ones relevant for determining the 
size of $C_v(H_{\ts}^{\cR})$. In fact, up to relabellings, we can reconstruct 
these tuples and the relevant component sizes of $F$ \emph{without} looking at 
the vertex labels (the tree structure, including the node types, is enough). 
Motivated by this we say that $S_{j,w}$ is \emph{bad} if one of the following 
conditions hold:
\begin{itemize}
	\item $S_{j,w}$ contains some tuple $\uu =(u_1, \ldots, u_\ell)$ 
multiple times.
	\item $S_{j,w}$ contains a tuple $\uu =(u_1, \ldots, u_\ell)$ where 
$u_i$ with $i \neq j$ is already a vertex node of $\cTv$ constructed so far. 
	\item $S_{j,w}$ contains a tuple $\uu =(u_1, \ldots, u_\ell)$ where 
$u_i$ and $u_k$ with $i \neq k$ are in the same component of $F$ (note that this 
holds for $u_i \in C_w(F)$ for $i \neq j$). 
	\item $S_{j,w}$ contains tuples $\uu =(u_1, \ldots, u_\ell)$ and 
$\uv =(v_1, \ldots, v_\ell)$ for which $u_i$ and $v_k$ with 
$i,k \neq j$ are in the same component of $F$. 
\end{itemize}
Otherwise $S_{j,w}$ is \emph{good}; Observe that if $S_{j,w}$ is good, then 
in $\cTv$ none of $w$'s component node descendants $u_i$ with 
$\uu =(u_1, \ldots, u_\ell) \in S_{j,w}$ are ignored. For this reason we call 
$\cTv$ \emph{good} if every $S_{j,w}$ is good. 
In the following we estimate the probability that $S_{j,w}$ is \emph{bad}. 
Clearly, there are at most $n^{\ell-1}$ different tuples with $u_j=w$. 
Recalling that $\cVv$ denotes the vertex nodes of $\cTv$, there are at most 
$\ell n^{\ell-2} |\cVv|$ different tuples satisfying the second 
condition, and at most $\ell^2 n^{\ell-2} |L_1(F)|$ tuples to which the 
third condition applies. Similarly, there are at most 
$\ell^2 n^{2(\ell-2)+1} |L_1(F)|$ pairs of tuples which satisfy 
the last condition. 
Recall that the random variables $A_{\uu}(\ts)$, which count the number of 
times $\uu$ is in $E_{\ts}$, are independent for different tuples $\uu$. 
So, using \eqref{eq:Auv2}, \eqref{eq:Auv1} and $\ts \le 1$, whenever 
$\max\{|\cVv|,|L_1(F)|\} \le U$ holds we see that the probability of 
$S_{j,w}$ being bad is at most 
\begin{equation}\label{eq:bad:Sjw}
n^{\ell-1} \cdot \ts^2 /n^{\ell} + 2\ell^2 n^{\ell-2}U \cdot \ts /n^{\ell-1} + \ell^2 n^{2\ell-3}U \cdot \ts^2/n^{2(\ell-1)} \le 4\ell^2 U/n . 
\end{equation}

\begin{figure}[t]
\centering
  \setlength{\unitlength}{1bp}%
  \begin{picture}(260.0, 89.00)(0,0)
  \put(0,0){\includegraphics{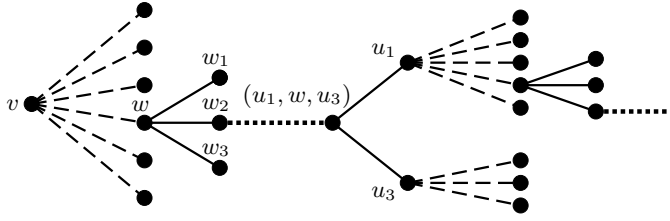}}
  \put(56.08,44.08){\fontsize{8.54}{10.24}\selectfont \makebox[0pt]{$w$}}
  \put(7.80,44.65){\fontsize{8.54}{10.24}\selectfont \makebox[0pt]{$v$}}
  \put(147.66,65.52){\fontsize{8.54}{10.24}\selectfont \makebox[0pt]{$u_1$}}
  \put(147.66,11.59){\fontsize{8.54}{10.24}\selectfont \makebox[0pt]{$u_3$}}
  \put(114.99,47.62){\fontsize{8.54}{10.24}\selectfont \makebox[0pt]{$(u_1,w,u_3)$}}
  \put(84.43,61.94){\fontsize{8.54}{10.24}\selectfont \makebox[0pt]{$w_1$}}
  \put(84.43,44.93){\fontsize{8.54}{10.24}\selectfont \makebox[0pt]{$w_2$}}
  \put(84.43,27.92){\fontsize{8.54}{10.24}\selectfont \makebox[0pt]{$w_3$}}
  \end{picture}%
	\caption{\label{fig:Tv} Example of the exploration tree $\cTv$ for $\ell=3$. 
The children of the root vertex $v$ are $w \in C_v(F)$ (vertex nodes), which in 
turn each have children $w_1, w_2, w_{3}$ (index nodes of types $1, 2, 3)$. 
Every $w_j$ has all (so far untested) tuples $\uu=(u_1,u_2,u_3) \in E_{\ts}$ 
with $u_j=w$ as children (tuple nodes), whose descendants are component nodes 
$u_i$ with $i \neq j$ (of types $1,2$). If $u_i$ is not already a vertex node 
of $\cTv$, then its children are $w \in C_{u_i}(F)$ (vertex nodes), for which 
we repeat the above construction. }
\end{figure}

To understand the structural properties of $\cTv$ it will be 
useful to compare it with a closely related process that is simpler to analyze. 
Recall that when determining the $S_{j,w}$ we only consider so far untested 
tuples. Thus each $S_{j,w}$ is dominated (with respect to the subset relation) 
by $\St_{j,w}$, where for each of the $n^{\ell-1}$ tuples 
$\uu =(u_1, \ldots, u_\ell)$ with $u_j=w$, independently, the number 
of its arrivals is given by a $\Po(\ts/n^{\ell-1})$ distribution.
There is a natural coupling between $S_{j,w}$ and $\St_{j,w}$ 
which only fails if $\St_{j,w}$ contains $\uu $ which are forbidden
for $S_{j,w}$. Since each of these `bad' tuples contains at least one vertex 
from $\cVv$, there are at most $\ell |\cVv| n^{\ell-2}$ of them. 
So, with \eqref{eq:Auv1} and $\ts \le 1$ in mind, by considering the 
probability that $\St_{j,w}$ selects at least one of them, whenever 
$|\cVv| \le U$ holds it follows that 
\begin{equation}
\label{eq:dtv:Sjw}
\dtv{{S_{j,w}}}{{\St_{j,w}}} \le \ell |\cVv| n^{\ell-2} \cdot \ts / n^{\ell-1} \le \ell U/n . 
\end{equation}
We now define $\tpv=\tpv(F)$ similarly to $\cTv$: we employ the same 
construction except that we use (independent copies of) $\St_{j,w}$ instead of 
$S_{j,w}$ and always proceed as if $\St_{j,w}$ is good. Since each $S_{j,w}$ is 
dominated by (may be regarded as a subset of) $\St_{j,w}$, it follows that 
$\cTv$ is dominated by $\tpv$ with respect to the subgraph relation. 
Denoting the set of vertex nodes of $\tpv$ by $\tpvV$, we see that $\cVv$ 
is dominated by $\tpvV$.

The next lemma states that the number of vertex nodes in $\cTv$ and $\tpv$ have 
(uniform) exponential decay. 
\begin{lemma}\label{lem:Tail:1}
Suppose that \eqref{eq:Nk:sigma} and \eqref{eq:NK:tailsum}--\eqref{eq:NK:chi} 
hold with $\beta > 1$. 
There exist $a,A > 0$ (depending only on $\ell,L,\sigma,\beta,B$) such that 
for all $0 \le \ts \le \sigma$ and $s \ge 0$ we have 
$\Pr(|\cVv| \ge s) \le \Pr(|\tpvV| \ge s) \le Ae^{-as}$, 
$\E N_{\ge s}(H_{\ts}^{\cR}) \le Ae^{-as} n$ and 
$\Pr(L_1(H_{\ts}^{\cR}) \ge s) \le Ae^{-as} n$. 
\end{lemma} 
Before giving the proof of this result, which is based on branching processes 
arguments, we use it to show that $\cTv$ and $\tpv$ can be coupled so that 
they typically agree. Note that at distance $4i+1, 4i+2, 4i+3,4i+4$ from the 
root $\cTv$ and $\tpv$ always have vertex, index, tuple and component nodes. 
\begin{lemma}\label{lem:Cpl:1}
Suppose $n \ge n_0(\ell,L,\sigma,\beta,B)$ and that the assumptions of 
Theorem~\ref{thm:evolution} as well as \eqref{eq:ts} hold. 
There exists a coupling of $\cTv$ and $\tpv$ so that with probability at 
least $1-(\log n)^{4}/n$ we have $\cTv = \tpv$ and $\cTv$ is good. 
\end{lemma}
\begin{proof}%
We write $T^i$ for the \emph{restriction} of a rooted tree $T$ to all vertices 
within distance at most $i$ from the root. Let $\cVv^i$ and $\tpvV^i$ denote 
the vertex nodes in $\cTv^i$ and $\tpv^i$, respectively. 
Recall that $U=(\log n)^{6/5}$. Since $\beta > 1$, note that $L_1(F) \le U$ 
follows from \eqref{eq:NK:tailsum} for $n \ge n_0(B,\beta)$.

We inductively couple $\cTv^{4i+1}$ and $\tpv^{4i+1}$ for $0 \le i \le U$ 
so that with probability at least $1-i \cdot 5\ell^3 U^2 /n$ we have either 
$\max\{|\cVv^{4i+1}|,|\tpvV^{4i+1}|\} \ge U$, or $\cTv^{4i+1}=\tpv^{4i+1}$ with 
all $S_{j,w}$ of $\cTv^{4i+1}$ being good. 
The base case $i=0$ is straightforward, as both use the same procedure for 
generating the root and its children. 
Now suppose that we have constructed $\cTv^{4i+1}$ and $\tpv^{4i+1}$ coupled 
as above. In the following we sequentially consider vertex nodes $w$ at 
distance $4i+1$ from the root and extend the coupling to their descendants with 
distance up to $4(i+1)+1$; here we clearly may assume 
$|\cVv^{4i+1}|=|\tpvV^{4i+1}| < U$. For each vertex node $w$ we create $\ell$ 
index nodes $w_1, \ldots w_{\ell}$ (of types $1, \ldots, \ell$). We abandon our 
coupling whenever we have found more than $U$ vertex nodes (in which case we 
are done), so \eqref{eq:dtv:Sjw} holds. Thus we can couple $S_{j,w}$ and 
$\St_{j,w}$ so that they agree with probability at least $1-\ell U/n$. Now we 
also abandon our coupling whenever $S_{j,w}$ is bad, which happens with 
probability at most $4 \ell^2 U/n$ by \eqref{eq:bad:Sjw}. The point is that 
given good $S_{j,w}=\St_{j,w}$, in both cases the same deterministic 
construction is used for generating the descendants of $w_j$ with distance up 
to $4(i+1)+1$ from the root. 
So, by repeating this for $w_1, \ldots, w_{\ell}$, with probability at least 
$1-5 \ell^3 U/n$ we can couple the descendants of $w$ with distance up to 
$4(i+1)+1$ from the root. 
Since we follow this argument for each of the at most $U$ vertex nodes at 
distance $4i$ from the root, we see that we can extend our coupling to 
$\cTv^{4(i+1)+1}$ and $\tpv^{4(i+1)+1}$ with probability at least 
$1-5\ell^3 U^2 /n$, establishing the claim.

Finally, by Lemma~\ref{lem:Tail:1} we know that $V=\max\{|\cVv|,|\tpvV|\} < U/10$ 
holds with probability at least, say, $1-n^{-9}$ for $n \ge n_0(a,A)$. This 
together with the above coupling completes the proof (as there are no vertex 
nodes with distance larger than $4V+1$ from the root). 
\end{proof}

We now introduce an idealized `infinite' version $\bp[{\varphi,t}]$ of the 
exploration tree that is defined \emph{without} reference to $n$ or $F$, and 
in which `bad' things (such as `ignored' component nodes) cannot happen by 
definition. Let $R$ be the random variable with $\Pr(R=k) = \varphi(k)$ for 
each $k \ge 1$, where $\varphi$ is given by Theorem~\ref{thm:evolution}. 
We start $\bp[{\varphi,t}]$ with a root node and add $R$ vertex nodes as 
children. Then, given any vertex node, we deterministically create $\ell$ 
children (index nodes of types $1, \ldots, \ell$). Each of these, 
independently, has $Z \sim \Po(t)$ children (tuple nodes). For each of these 
grandchildren we assign again (deterministically) $\ell-1$ children (component 
nodes of types $1, \ldots, \ell-1$). All of these, independently, give birth 
to $R$ many descendants (vertex nodes).

For our subsequent analysis it will be key to observe that if we are only 
interested in equality up to isomorphisms, then we can generate $\tpv$ in 
a more convenient way, similarly to $\bp$. Indeed, using standard properties of 
Poisson processes and noting that selecting a uniform tuple 
$\uu =(u_1, \ldots, u_\ell)$ with $u_j=w$ is equivalent to picking $\ell-1$ 
random vertices, we can generate the descendants of $w_j$ constructed by 
$\St_{j,w}$ using the following three-generation tree process: the root has 
$Z \sim \Po(\ts)$ children (tuple nodes); then for each of the resulting 
children we construct (deterministically) $\ell-1$ grandchildren (component 
nodes of types $1, \ldots, \ell-1$), which each in turn give birth to $N$ 
descendants (vertex nodes), where $N \sim |C_u(F)|$ for a uniformly and 
independently chosen vertex $u$. Comparing the resulting construction with 
$\bp$, it follows that we can generate $\tpv$ up to relabellings in the same 
way as $\bp$, with the only difference that we use $N$ instead of $R$. 
\begin{proof}[Proof of Lemma~\ref{lem:Tail:1}]
Since $\cVv$ is dominated by (may be regarded as a subset of) $\tpvV$, 
we have $\Pr(|\cVv| \ge s) \le \Pr(|\tpvV| \ge s)$. 
Using this inequality, we claim that it is enough to prove existence of 
$a,A > 0$ (depending only on $\ell,L,\sigma,\beta,B$) satisfying 
\begin{equation}
\label{eq:tail}
\Pr(|\tpvV| \ge s) \le Ae^{-as} \quad \text{for all $s \ge 0$. }
\end{equation}
Indeed, recall that $v$ is chosen uniformly at random, so that 
$\Pr( |C_v(H_{\ts}^{\cR})| \ge s \mid H_{\ts}^{\cR}=G) = N_{\ge s}(G)/n$. 
Taking expectations, we see that $\E N_{\ge s}(H_{\ts}^{\cR}) = n \Pr( |C_v(H_{\ts}^{\cR})| \ge s)$. 
Using \eqref{eq:Vv} we have $|C_v(H_{\ts}^{\cR})| \le |\cVv|$, 
so $\Pr( |C_v(H_{\ts}^{\cR})| \ge s) \le Ae^{-as}$ by \eqref{eq:tail}. 
Now Markov's inequality gives $\Pr(L_1(H_{\ts}^{\cR}) \ge s) \le Ae^{-as} n$.

In the remainder we establish \eqref{eq:tail} using 
Lemmas~\ref{lem:bp:twogeneration} and~\ref{lem:bp:tree:Z}. Let $Z_j$ be 
independent copies of $Z \sim \Po(\ts)$, and let $v_{j,r,k}$ be uniformly 
and independently chosen random vertices. We henceforth construct $\tpv$ up to 
relabellings, as described in the paragraph proceeding this proof. Given a 
vertex node $w$ with distance $4i+1$ from the root, in this tree construction 
it has 
\[
W = \sum_{1 \le j \le \ell} \sum_{1 \le r \le Z_{j}} \sum_{1 \le k \le \ell-1} |C_{v_{j,r,k}}(F)| 
\]
vertex node descendants at distance $4(i+1)+1$ from the root, where 
$\E(W)=\ell \ts (\ell-1) S(F) \le \ell \sigma (\ell-1) L < 1$ 
due to $\ts \le \sigma$ and \eqref{eq:Nk:sigma}. 
Note that $F_W(z)=[F_{Z}([F_N(z)]^{\ell-1})]^\ell$, where 
$N \sim |C_u(F)|$ for a uniformly chosen vertex $u$. 
By \eqref{eq:NK:tailsum} we have $[F_N(\beta)]^{\ell-1} \le B^{\ell-1}$. 
Now, since $Z \sim \Po(\ts)$ and $0 \le \ts \le \sigma$, it easily 
follows that $F_{Z}(z) = e^{\ts(z-1)} \le e^{\sigma z}$ for $z \ge 0$, 
so $F_W(\beta) \le \tilde{B}=\tilde{B}(\ell,\sigma,B)$. 
Let $W^{+}$ be the size of the Galton--Watson branching process in which 
each node, independently, has $W$ children. Lemma~\ref{lem:bp:tree:Z} yields 
$F_{W^{+}}(\delta) \le D$, where $\delta > 1$ and $D>0$ depend only on 
$\ell,L,\sigma,\beta,\tilde{B}$. 
Since the distribution of $W$ does not depend on the $w$ or $i$ considered 
above, it in particular follows that each vertex node with distance $1$ from 
the root has $W^{+}$ vertex node descendants in $\tpv$.

Finally, note that $\tpv$ starts with a root vertex which gives birth to $N$ 
vertex node children, each of whose vertex nodes descendants is given by 
independent copies of $W^{+}$. With this in mind $|\tpvV| \sim T$, where 
$T$ is a two-generation branching process where the root has $N$ children, 
and then each of these, independently, has $W^{+}$ children. Recall that 
$F_N(\beta) \le \tilde{B}$ and $F_{W^{+}}(\delta) \le D$ for 
$\beta,\delta > 1$ and $\tilde{B},D > 0$. So, Lemma~\ref{lem:bp:twogeneration} 
yields \eqref{eq:tail} for $A=\tilde{B}$ and $a > 0$ depending only on 
$\beta,\delta,\tilde{B},D$. As explained, this completes the proof. 
\end{proof}

Recall that $\bp$ uses the same construction as $\tpv$, with the difference 
that it employs $R$ instead of $N$. 
When establishing the exponential decay in the proof of Lemma~\ref{lem:Tail:1}, 
note that the only properties of $N$ used are $\E N = S(F) \le L$ and 
$F_N(\beta) \le B$. Since $\E R = \chi(\varphi) \le L$ and 
$F_R(\beta) \le B$ by \eqref{eq:varphi:technical}--\eqref{eq:varphi:chi}, the 
same argument thus carries over word-by-word when applied to the vertex nodes 
of $\bp$, which we denote by $\bpV$. 
\begin{lemma}\label{lem:Tail:2}
Suppose that \eqref{eq:varphi:technical}--\eqref{eq:varphi:chi} and 
\eqref{eq:Nk:sigma} hold with $\beta > 1$. 
There exist $a,A > 0$ (depending only on $\ell,L,\sigma,\beta,B$) such that 
for all $0 \le \ts \le \sigma$ and $s \ge 0$ we have 
$\Pr(|\bpV| \ge s) \le Ae^{-as}$, where $a,A$ are defined in the same way 
as in Lemma~\ref{lem:Tail:1}. 
\nopf
\end{lemma}
After these preparations, we are now ready to show that we can couple $\cTv$ 
and $\bp$ so that they typically agree up to isomorphisms (by using $\tpv$ 
as an `intermediate' process). 
\begin{lemma}\label{lem:Cpl:2}
Suppose $n \ge n_0(\ell,L,\sigma,\beta,B)$ and that the assumptions of 
Theorem~\ref{thm:evolution} as well as \eqref{eq:ts} hold. 
There exists a coupling of $\cTv$ and $\bp$ so that with probability at 
least $1-(\log n)^{C+5} n^{-1/2}$ we have $\cTv \cong \bp$ and $\cTv$ is good. 
\end{lemma}
\begin{proof}
Recall that $U=(\log n)^{6/5}$. 
By Lemma~\ref{lem:Cpl:1} it suffices to couple $\tpv$ and 
$\bp$ so that with probability at least $1-4 \ell^2 U^4(\log n)^{C} n^{-1/2}$ we 
have $\tpv \cong \bp$. To this end we use a similar but simpler argument as in 
the proof of Lemma~\ref{lem:Cpl:1}, inductively extending our coupling from 
distance $4i+1$ to $4(i+1)+1$ from the root. As before, using Lemma~\ref{lem:Tail:1} 
and~\ref{lem:Tail:2} we can safely abandon our coupling whenever we have seen 
at least $U$ vertex nodes, or when we reach distance $U$ from the root. 
In the inductive step, the only difference between $\bp$ and $\tpv$ is that 
$\bp$ uses $R$ whereas $\tpv$ uses $N$. Recall that $\Pr(R=k) = \varphi(k)$ and 
$\Pr(N=k) = N_k(F)/n$. It is not difficult to see that \eqref{eq:varphi:tailsum} 
and \eqref{eq:NK:tailsum} imply $\Pr(R \ge U) \le n^{-2}$ and 
$\Pr(N \ge U) = 0$ for $n \ge n_0(\beta,B)$. Using these tail estimates 
together with \eqref{eq:Nk:approx}, by distinguishing values smaller and larger 
than $U$ we obtain 
\begin{equation}
\label{eq:dtv:RN}
\dtv{R}{N} \le U \cdot (\log n)^C n^{-1/2} + n^{-2} \le 2 U (\log n)^{C} n^{-1/2} . 
\end{equation}
We furthermore may safely abandon our coupling whenever some index node has 
$Z \ge U$ children, since (using $\ts \le \sigma$) this occurs with 
probability at most $n^{-9}$ for $n \ge n_0(\sigma)$. The point is that this 
ensures that we only need to couple $R$ and $N$ at most $\ell^2 U^2$ times when 
going from distance $4i+1$ to $4(i+1)+1$. So, each time we can extend the 
coupling inductively with probability at least, say, 
$1-3 \ell^2 U^3 (\log n)^{C} n^{-1/2}$. Arguing as in the proof of 
Lemma~\ref{lem:Cpl:1}, this completes the coupling argument. 
\end{proof}

\subsubsection{Expected component sizes} 
After analyzing the tuple and component structure induced by $E_{\ts}$, we 
now consider the second exposure round, where the selected tuples are 
presented in random order to $\cR$. Intuitively, the coupling given by 
Lemma~\ref{lem:Cpl:2} allows us to estimate $\E N_k(H_{\ts}^{\cR})$ 
using $\bp$. As we shall see, this also carries over to 
$\E N_k(F_{\ts n}^{\cR})$.

Recall that if the exploration tree $\cTv \cong T$ is good, then during its 
construction no component nodes are ignored. As mentioned in 
Section~\ref{sec:CEP}, the key point is that if no nodes are ignored (i.e., 
all component nodes have at least one child), then from the structure of $T$ 
(which includes the vertex types) we can reconstruct all tuples in $E_{\ts}$ 
and component sizes of $F$ (up to relabellings) which are relevant for 
determining $|C_v(H_{\ts}^{\cR})|$. We denote the corresponding set of tuples 
and component sizes by $\cT_T$ and $\cC_T$, respectively. 
As the above `reconstruction' procedure only uses the tree-structure of $T$, 
it in fact can be applied to any exploration tree in which each component node 
has at least one child; so, in particular, to $\bp\cong T$. 
In the following we define $|C^\cR(T)|$ for any exploration tree $T$, where 
we formally set $|C^\cR(T)|=0$ if $T$ contains a component node with $0$ 
descendants. Otherwise, we traverse in (uniform) random order the tuples in 
$\cT_T$; for each tuple we present the component sizes of its vertices to $\cR$ 
and update the list of components (and their sizes) according to the decisions 
of $\cR$ (by adding the pairs selected by $\cR$). 
Finally, we define $|C^\cR(T)|$ as the size of the resulting component 
which contains the root vertex of $T$. 
Since the second exposure round of $H_{\ts}^{\cR}$ presents the tuples in 
$E_{\ts}$ to $\cR$ in random order, a moment's thought reveals that 
conditional on $\cTv \cong T$ being good, both $|C_v(H_{\ts}^{\cR})|$ and 
$|C^\cR(\cTv)|$ have exactly the same distribution for size rules. 
So, for all $k \ge 1$ we have 
\begin{equation}\label{eq:C}
\Pr( |C_v(H_{\ts}^{\cR})|=k \mid \text{$\cTv \cong T$ is good}) = \Pr( |C^\cR(\cTv)|=k \mid \text{$\cTv \cong T$ is good}) .
\end{equation}
Before using this observation to estimate $\E N_k(H_{\ts}^{\cR})$, we first 
collect some basic properties of the function $\varrho$, where we set 
\begin{equation}\label{eq:varrho}
\varrho(k,t)=\Pr(|C^\cR(\bp[{\varphi,t}])|=k) \quad \text{for all $(k,t) \in \NN \times \Rp$.}
\end{equation}
\begin{lemma}\label{lem:varrho:prop}
Suppose that \eqref{eq:varphi:technical}--\eqref{eq:varphi:chi} and 
\eqref{eq:Nk:sigma} hold with $\beta > 1$. 
The function $\varrho:\NN \times \Rp \to [0,1]$ defined in \eqref{eq:varrho} 
depends only on $\varphi,\cR,\ell$ and satisfies $\sum_{k \ge 1} \varrho(k,\ts) = 1$ 
for all $0 \le \ts \le \sigma$. 
Furthermore, there exist $a,A > 0$ (depending only on $\ell,L,\sigma,\beta,B$) 
such that for all $0 \le \ts \le \sigma$ and $s \ge 0$ we have 
$\varrho(s,\ts) \le Ae^{-as}$, where $a,A$ are given by Lemma~\ref{lem:Tail:2}. 
\end{lemma}
\begin{proof}
The definitions of $C^\cR(\cdot)$ and of $\bp[{\varphi,t}]$ depend only on 
$\cR,\ell$ and on $\varphi,\cR,\ell,t$ respectively. So, from \eqref{eq:varrho} 
we see that $\varrho:\NN \times \Rp \to [0,1]$ depends only on 
$\varphi,\cR,\ell$. 
Since the component containing the root vertex of $\bp$ can only 
contain vertex nodes of $\bp$, we see that $1 \le |C^\cR(\bp)| \le |\bpV|$ 
holds, from which $\varrho(0,\ts)=0$ follows. Furthermore, Lemma~\ref{lem:Tail:2} 
implies $\varrho(s,\ts) \le \Pr(|\bpV| \ge s) \le Ae^{-as}$ for all $s \ge 1$, 
where $a,A > 0$ depend only on $\ell,L,\sigma,\beta,B$. Similarly, for all 
$s \ge 0$ we have $\Pr(\text{$\bp$ is infinite}) \le \Pr(|\bpV| \ge s) \le Ae^{-as}$. 
But $Ae^{-as} \to 0$ as $s \to \infty$, so $\Pr(\text{$\bp$ is infinite})=0$, 
which in turn yields $\sum_{k \ge 1} \varrho(k,\ts) = 1$. 
\end{proof}
\begin{lemma}\label{lem:Nk:E}
Suppose $n \ge n_0(\ell,L,\sigma,\beta,B)$ and that the assumptions of 
Theorem~\ref{thm:evolution} as well as \eqref{eq:ts} hold. We have 
\begin{equation}\label{eq:Nk:E}
\E N_k(H_{\ts}^{\cR}) = \varrho(k,\ts) n \pm (\log n)^{C+6} n^{1/2} \quad \text{for all $k \ge 1$.}
\end{equation}
\end{lemma}
\begin{proof}
Similar as in the proof of Lemma~\ref{lem:Tail:1}, since $v$ is chosen 
uniformly at random we have 
$\E N_k(H_{\ts}^{\cR}) = n \Pr( |C_v(H_{\ts}^{\cR})|=k )$. 
To prove the claim it thus suffices to relate $\Pr( |C_v(H_{\ts}^{\cR})|=k )$ 
and $\varrho(k,\ts)=\Pr(|C^\cR(\bp)|=k)$. The coupling of Lemma~\ref{lem:Cpl:2} 
implies that $\bp \cong \cTv$ holds with probability at least 
$1-(\log n)^{C+5} n^{-1/2}$ for $n \ge n_0(\ell,L,\sigma,\beta,B)$. Hence
\[
\Pr( |C^\cR(\cTv)|=k ) = \Pr( |C^\cR(\bp)|=k ) \pm 2(\log n)^{C+5} n^{-1/2} .
\]
Since this coupling also implies that $\cTv$ is good, using \eqref{eq:C} it 
follows that 
\[
\Pr( |C_v(H_{\ts}^{\cR})|=k ) = \Pr( |C^\cR(\cTv)|=k ) \pm 2(\log n)^{C+5} n^{-1/2}. 
\]
Finally, combining our findings and recalling \eqref{eq:varrho}, we readily 
obtain \eqref{eq:Nk:E}. 
\end{proof}
Now we relate $H_{\ts}^{\cR}$ with $F_{\ts n}^{\cR}$ by establishing that 
$\E N_k(H_{\ts}^{\cR}) \approx \E N_k(F_{\ts n}^{\cR})$. 
\begin{lemma}\label{lem:Nk:E:transfer}
Suppose that $0 \le \ts \le 1$. Then for $n \ge n_0(\ell)$ we have 
\begin{equation}
\label{eq:Nk:E:transfer}
\E N_k(F_{\ts n}^{\cR}) = \E N_k(H_{\ts}^{\cR}) \pm k (\log n) n^{1/2} \quad \text{for all $k \ge 1$.}
\end{equation}
\end{lemma}
\begin{proof}
Observe that $N_k$ changes by at most $\ell k$ per step. So, for $r \le s$ we 
have $\E(N_k(F_s^{\cR}) \mid F_r^{\cR}=G) = N_k(G) \pm (s-r) \ell k$. 
Taking expectations and restricting our attention to 
$r \in \{\ts n-i,\ts n\}$ shows that for each $i \ge 0$ we have 
\begin{equation}
\label{eq:Nk:transfer:Gm:liptschitz}
\E N_k(F_{\ts n \pm i}^{\cR}) = \E N_k(F_{\ts n}^{\cR}) \pm \ell k i .
\end{equation}
Set $s=3\sqrt{n \log n}$. Using $\ts \le 1$, standard Chernoff bounds 
yield that $|E_{\ts}| = \ts n \pm s$ with probability at least $1-n^{-2}$ 
for $n \ge n_0$. Combining this with \eqref{eq:Nk:transfer:Gi} and 
\eqref{eq:Nk:transfer:Gm:liptschitz}, we readily obtain 
\[
\E N_k(H_{\ts}^{\cR}) = \E N_k(F_{\ts n}^{\cR}) \pm \ell k s \pm n \cdot n^{-2} ,
\]
which implies \eqref{eq:Nk:E:transfer} for $n \ge n_0(\ell)$, with room to 
spare. 
\end{proof}

\subsubsection{Concentration of component sizes}
In this section we establish concentration of $N_k(F_i^{\cR})$ around its 
expected value. The main technical difficulty here is that few changes of the 
offered tuples might alter many decisions of size rules (as the component sizes 
observed in later rounds can change); as we shall see, the bounds for 
$L_1(\cdot)$ implied by Lemma~\ref{lem:Tail:1} will be a crucial ingredient for 
showing that this is typically not the case. 
\begin{lemma}\label{lem:Nk:C}
Suppose $n \ge n_0(\ell,L,\sigma,\beta,B)$ and that the 
assumptions of Theorem~\ref{thm:evolution} hold. 
With probability at least $1-n^{-250}$, for every $0 \le i \le \sigma n$ 
we have 
\begin{equation}
\label{eq:Nk:concentration}
N_k(F_i^{\cR}) = \E N_k(F_i^{\cR}) \pm (\log n)^2 n^{1/2} \quad \text{for all $1 \le k \le (\log n)^2$.}\\
\end{equation}
\end{lemma}
\begin{proof}
We sequentially draw $\sigma n$ random tuples and consider two associated 
graph sequences $F_i^{\cR}$ and $F_i^{\cI}$, where the `influence' rule $\cI$ 
in each step simply joins all $\ell$ randomly chosen vertices by edges. Note 
that $F_i^{\cR} \subseteq F_i^{\cI}$ always holds. Let $\cL$ denote the event 
that $L_1(F_{\sigma n}^{\cI}) < U=(\log n)^{6/5}$, which by monotonicity 
implies $L_1(F_i^{\cI}) < U$ for all $0 \le i \le \sigma n$. Combining 
Lemma~\ref{lem:Tail:1} with \eqref{eq:pittel}, for $n \ge n_0(a,A,\sigma)$ we 
have, say, 
\begin{equation}
\label{eq:Nkc:L}
\Pr(\neg\cL) \le 3 \sqrt{\sigma n} \cdot \Pr(L_1(H_{\sigma}^{\cI}) \ge U) \le n^{-300} . 
\end{equation} 
For every $1 \le i \le \sigma n$ let $X_{k,i}$ denote the number of vertices 
which satisfy $|C_v(F_i^{\cR})|=k$ and $|C_v(F_i^{\cI})| < U$. When $\cL$ holds 
no vertices are `ignored' due to $|C_v(F_i^{\cI})| \ge U$, so we have 
$X_{k,i} = N_{k}(F_i^{\cR})$. Together with \eqref{eq:Nkc:L} this readily 
gives, say, $\E X_{k,i} = \E N_{k}(F_i^{\cR}) \pm n^{-1}$. So, for 
$\Delta = U^{3/2}n^{1/2}$ it follows that 
\begin{equation}
\label{eq:Nkc:I}
\Pr(\{|N_k(F_i^{\cR})-\E N_k(F_i^{\cR})| \ge 2 \Delta\} \cap \cL) \le \Pr(|X_{k,i}-\E X_{k,i}| \ge \Delta) .
\end{equation} 
Note that for every size rule $\cR$ the random variable $X_{k,i}$ can be 
written as $X_{k,i}=f(\uv_{1}, \ldots, \uv_{i})$, where the 
$\uv_{j}$ denote the $\ell$-tuples generated by the $\ell$-vertex 
process in each step (uniformly and independently). We claim that the function 
$f$ satisfies $|f(\omega)-f(\tilde{\omega})| \le 4\ell U$ whenever $\omega$ 
and $\tilde{\omega}$ differ in one coordinate, i.e., in one tuple. 
Assuming that $(\uv_{1}, \ldots, \uv_{i})$ yield 
$F_i^{\cR}$ and $F_i^{\cI}$, respectively, let $\tilde{F}_i^{\cR}$ and 
$\tilde{F}_i^{\cI}$ denote the graphs which result by changing 
$\uv_{j}$ to $\underline{\tilde{v}}_{j}$. 
Since $F_i^{\cI}$ and $\tilde{F}_i^{\cI}$ only differ in the edges induced by 
$\uv_{j}$ and $\underline{\tilde{v}}_{j}$, there is a set of vertices 
$W$ containing at most $2\ell$ components in each of $F_i^{\cI}$ and 
$\tilde{F}_i^{\cI}$ so that outside of $W$ the component structure of both 
graphs is the same (to see this note that the order is irrelevant for $\cI$, 
so we may assume $i=j$; then defining $W$ as the union of the components 
containing the vertices of $\uv_{j}$ and $\underline{\tilde{v}}_{j}$ 
in $F_{i-1}^{\cI}=\tilde{F}_{i-1}^{\cI}$ suffices). 
The key point is now that for size rules the decisions of $\cR$ in $F_i^{\cR}$ 
and $\tilde{F}_i^{\cR}$ are the same for all tuples which contain no vertices 
from $W$ (indeed, if a decision of $\cR$ is modified then any changes of the 
resulting component sizes can only `propagate' inside the components of 
$F_i^{\cI}$ and $\tilde{F}_i^{\cI}$; so only tuples containing vertices from 
$W$ can be affected). 
It follows that the component structure outside of $W$ is also the same in 
$F_i^{\cR}$ and $\tilde{F}_i^{\cR}$. Recall that $W$ contains at most $2\ell$ 
components in each of $F_i^{\cI}$ and $\tilde{F}_i^{\cI}$. 
So, since $X_{k,i}$ only counts those vertices $v$ with $|C_v(F_i^{\cI})| < U$, 
we see that a change of one tuple can alter $f$ by at most $2 \cdot 2\ell \cdot U$, 
as claimed. So, recalling that $1 \le i \le \sigma n$, for 
$n \ge n_0(\ell,\sigma)$ McDiarmid's inequality~\cite{McDiadmid1989} implies 
\begin{equation}
\label{eq:Nkc:AH}
\Pr(|X_{k,i}-\E X_{k,i}| \ge \Delta) \le \exp\left(-\frac{2\Delta^2}{i (4\ell U)^2}\right) \le n^{-300} \enspace . 
\end{equation} 
Finally, after combining \eqref{eq:Nkc:L}--\eqref{eq:Nkc:AH}, taking a union 
bound to account for all choices of $1 \le i \le \sigma n$ and 
$1 \le k \le (\log n)^2$ completes the proof (noting that the claim is trivial 
for $i=0$). 
\end{proof}
Using the main idea of the above proof we can directly show that 
$\E N_k(F_i^{\cR})$ is essentially independent of the initial graph 
$F_0^{\cR}=F$ for $i \le \sigma n$: for any two graphs $F,\tilde{F}$ satisfying 
the assumptions of Theorem~\ref{thm:evolution} their expected values can 
differ by at most, say, $(\log n)^{C+3} n^{1/2}$. 
The key point is that for such graphs we can construct a bijection $\Psi$ 
between their vertex sets which, up to an exceptional set $W$ of at most, say, 
$4U (\log n)^C n^{1/2}$ vertices, preserves the component structure of 
$F$ and $\tilde{F}$, respectively. Now, using $\Psi$ we couple 
$F_i^{\cR},F_i^{\cI}$ and $\tilde{F}_i^{\cR},\tilde{F}_i^{\cI}$ in a measure 
preserving way. Since changes can only propagate inside the components of the 
`influence' graphs, only those vertices whose components in $F_i^{\cI}$ or 
$\tilde{F}_i^{\cI}$ contain vertices of $W$ or $\Psi(W)$ can be `spoiled'. 
Intuitively, since the components usually have size at most $U$, under this 
coupling $N_k$ thus typically differs by at most $2|W| \cdot U$ for both 
graphs. Taking the error probability of 
$\max\{L_1(F_{\sigma n}^{\cI}),L_1(\tilde{F}_{\sigma n}^{\cI})\} < U$ into 
account, the claim now follows without much work.

\subsubsection{Putting things together} 
In this section we combine our findings to prove Theorem~\ref{thm:evolution}. 
Lemma~\ref{lem:varrho:prop} easily implies the first part, i.e., existence of 
$\varrho:\NN\times \Rp \to [0,1]$ with the desired properties. Indeed, it 
ensures that for every $\sigma \ge 0$ satisfying \eqref{eq:Nk:sigma} there 
exist $a,A > 0$ (depending only on $\ell,L,\sigma,\beta,B$) such that for 
every $t \in [0,\sigma]$ we have $\sum_{k \ge 1} \varrho(k,t)=1$ and 
\begin{equation}
\label{eq:rho}
\varrho(s,t) \le Ae^{-as} \quad \text{for all $s \ge 0$.}
\end{equation}
For later usage we remark that Lemma~\ref{lem:Tail:1} holds for the 
same $a,A$. Let $D=300/a>0$ and 
$\tilde{\beta}=\min\{e^{a/2},e^{1/(4D)}\}>1$. Now, using \eqref{eq:rho} we see 
that 
\begin{equation}\label{eq:varrho:Tailsum}
\sum_{k \ge 1} \varrho(k,t) \tilde{\beta}^k \le A \sum_{k \ge 1} e^{-ak/2} = \tilde{B}-1 ,
\end{equation} 
with $1 < \tilde{B} < \infty$ depending only on $a,A$. Similarly, we obtain 
\begin{equation}\label{eq:varrho:Chi}
\chi(\varrho,t)=\sum_{k \ge 1} k \varrho(k,t) \le A \sum_{k \ge 1} k e^{-ak} = \tilde{L}-1 ,
\end{equation} 
with $1 < \tilde{L} < \infty$ depending only on $a,A$. 
Summarizing, equations \eqref{eq:varphi:technical}--\eqref{eq:varphi:chi} hold 
when $\beta,B,L,\varphi(\cdot)$ are replaced by 
$\tilde{\beta},\tilde{B},\tilde{L},\varrho(\cdot,t)$, with room to spare.

Turning to properties of $F_{i}^{\cR}$, from Lemmas~\ref{lem:Nk:E}--\ref{lem:Nk:C} 
it follows that with probability at least $1-n^{-250}$, for every 
$0 \le i \le \sigma n$ (by considering $\ts=i/n \in [0,\sigma]$) we have 
\begin{equation}\label{eq:Nk:C}
N_k(F_{i}^{\cR}) = \varrho(k,i/n) n \pm 3(\log n)^{C+6} n^{1/2} \quad \text{for all $1 \le k \le (\log n)^2$}
\end{equation}
for $n \ge n_0(\ell,L,\sigma,\beta,B)$. 
Recall that Lemma~\ref{lem:Tail:1} holds with the $a,A$ chosen above. By 
definition of $D$ it follows that, with probability at least $1-n^{-250}$, we 
have 
\begin{equation}
\label{eq:L1}
L_1(F_{\sigma n}^{\cR}) \le D \log n
\end{equation}
for $n \ge n_0(A)$. In the remainder we assume that \eqref{eq:Nk:C}--\eqref{eq:L1} 
hold. Recalling \eqref{eq:rho} and the definition of $D$, note 
that for all $k \ge D\log n$ and $0 \le i \le \sigma n$ we have, say, 
$|\varrho(k,i/n)| \le n^{-9}$ for $n \ge n_0(A)$. 
Using \eqref{eq:L1} it follows that 
\begin{equation*}
\label{eq:Nk:C:0}
N_k(F_{i}^{\cR}) = \varrho(k,i/n) n \pm (\log n)^{C} n^{1/2} \quad \text{for all $k \ge D\log n$ and $0 \le i \le \sigma n$. }
\end{equation*}
Together with \eqref{eq:Nk:C}, for every $0 \le i \le \sigma n$ this 
establishes \eqref{eq:Nk:approx} with $C,F,\varphi(\cdot)$ replaced by 
$\tilde{C},F_{i}^{\cR},\varrho(\cdot,i/n)$ for $n \ge n_0(D)$, where 
$\tilde{C}=C+9$. 
Now, using \eqref{eq:varrho:Tailsum} and \eqref{eq:Nk:C}--\eqref{eq:L1} we 
see that for every $0 \le i \le \sigma n$ we have (as $i/n \in [0,\sigma]$) 
\begin{equation*}\begin{split}
\sum_{k \in [n]} N_{k}(F_{i}^{\cR}) \tilde{\beta}^k &\le n \sum_{1 \le k \le D \log n} \varrho(k,i/n) \tilde{\beta}^s + 3D (\log n)^{C+6} n^{1/2} \sum_{1 \le k \le D \log n} \tilde{\beta}^k \\
&\le n \sum_{k \ge 1} \varrho(k,i/n) \tilde{\beta}^k + 3D^2 (\log n)^{C+7} n^{3/4} \le \tilde{B}n 
\end{split}\end{equation*}
for $n \ge n_0(C,D)$, which establishes \eqref{eq:NK:tailsum} with 
$\beta,B,F$ replaced by $\tilde{\beta},\tilde{B},F_{i}^{\cR}$. 
It remains to show that \eqref{eq:NK:chi}--\eqref{eq:S2} hold. Recall that 
$S(G)=\sum_{k \in [n]} k N_k(G)/n$. Now, assuming 
$n \ge n_0(a,A,D)$ and using \eqref{eq:varrho:Chi}--\eqref{eq:L1} 
similarly as above, for every $0 \le i \le \sigma n$ we have 
\begin{equation*}\begin{split}
S(F_{i}^{\cR})&= \sum_{k \ge 1} k \varrho(k,i/n) \pm A\sum_{k \ge D \log n} k e^{-ak} \pm 3D^2 (\log n)^{C+8} n^{-1/2} \\
&= \chi(\varrho,i/n) \pm 4D^2 (\log n)^{C+8} n^{-1/2} , 
\end{split}\end{equation*}
which establishes \eqref{eq:S2} for $n \ge n_0(D)$. 
Finally, recalling \eqref{eq:varrho:Chi}, it follows that \eqref{eq:NK:chi} 
holds with $L,F$ replaced by $\tilde{L},F_{i}^{\cR}$ for $n \ge n_0(C,D)$, 
which completes the proof of Theorem~\ref{thm:evolution}.

\section{When does $\tcx=\tc$?}\label{sec:disc}
In this section we discuss Conjecture~\ref{conj}, first
showing that it does hold for many size rules and then, in Section~\ref{sec:delay},
showing that it cannot be extended to general 
$\ell$-vertex rules, i.e., that the critical point where the susceptibility 
blows up need not always coincide with the percolation threshold.

\subsection{Rules with uniform random edges}\label{sec:unif}
It is well known, and not hard to check, that under suitable
assumptions the graph $F_{\theta n}$ given by adding
$\theta n$ independent and uniformly random edges to a given
$n$-vertex initial graph $F$ can be viewed as an instance
of the inhomogeneous random graph model of Bollob\'as, Janson and Riordan~\cite{BJR}.
To make this precise, consider instead the (multi-)graph $\tF_{\theta}$
obtained from $F$ by adding a Poisson number $\Po(2\theta/n)$ of copies of each
of the $\binom{n}{2}$ possible edges, with these numbers independent; we may then
ignore multiple edges, as we are only interested in the component structure.
Since $\Po(\theta (n-1))$ edges are added in total, and there will be few multiple
edges, $\tF_{\theta}$ and $F_{\theta n}$ are essentially interchangeable (one
may use domination arguments comparing them for different $\theta$ to make this precise).
Given two components $C_1$ and $C_2$ of $F$, the number of edges between
them in $\tF_{\theta}$ is Poisson with mean $|C_1||C_2|2\theta/n$.
Making (for a change) the $n$ dependence explicit, let $H_n$ be the
random graph whose vertices are the components of $F$, with an edge
between two vertices if these components are joined by an edge of $\tF_{\theta}$.
We say that a vertex of $H_n$ has \emph{type $k$} if the corresponding
component of $F$ has $k$ vertices. Then
the probability of an edge between a given type-$i$ vertex and a given type-$j$ vertex 
of $H_n$ is $1-e^{-2\theta ij/n}$, which is around $2\theta ij/n$ if $i$ and $j$ are
not too big, and the events that different edges are present are independent.

More precisely, let $\kappa(i,j)=2\theta ij$ for all positive integers $i$ and $j$.
Suppose that $\mu$ is a finite measure on $\ZZ^+$, i.e., that $\mu_k=\mu(\{k\})\ge 0$
for all $k$ and $0<\sum_{k \ge 1} \mu_k<\infty$.
Let $F=F_n$ be a random $n$-vertex starting
graph. Suppose that, for each fixed $k\ge 1$,
\begin{equation}\label{BJR1}
 \frac{N_k(F_n)}{kn} \pto \mu_k
\end{equation}
as $n\to\infty$, i.e., that $H_n$ has asymptotically $\mu_k n$ vertices of type $k$,
and that
\begin{equation}\label{BJR2}
 \sum_{k \ge 1} k\mu_k =1.
\end{equation}
Then one can use \eqref{BJR2}, the fact that $F_n$ has $n$ vertices
and \eqref{BJR1} to show that whenever $K(n)\to\infty$ we have
\begin{equation}\label{notail}
 N_{\ge K(n)}(F_n)/n\pto 0,
\end{equation}
and it follows that for any $A\subset \ZZ^+$ we have
\begin{equation}\label{vspace}
 \sum_{k\in A} \frac{N_k(F_n)}{kn}\pto \sum_{k\in A}\mu_k.
\end{equation}
In the terminology of \cite{BJR},
this means that the (random) sets of vertices of the graphs $H_n$, together
with their types, form a \emph{generalized vertex space} on the
\emph{generalized ground space} $(\ZZ^+,\mu)$. Taking $A=\ZZ^+$ in \eqref{vspace}, 
we have in particular that $|H_n|/n\pto \mu(\ZZ^+)\in (0,\infty)$.
By \eqref{BJR2}
the function $\kappa$ forms an integrable kernel on the ground space $(\ZZ^+,\mu)$,
with integral $2\theta$. Finally, the technical `graphicality' condition of \cite{BJR}
is met since $\tF_{\theta}$ has asymptotically $\theta n$ edges.
It follows that under these assumptions, the results of~\cite{BJR}
apply to $H_n$ (see Remark 2.4 there). The most important
of these results is \cite[Theorem 3.1]{BJR}, which tells
us that $H_n$ will whp contain a giant component (one with $\Theta(n)$ vertices)
if and only if $\norm{T_\kappa}>1$, where $T_\kappa$ is a certain integral
operator associated to $\kappa$. In particular,
if $\norm{T_\kappa}>1$ then there is some constant $\alpha=\alpha(\kappa,\mu)>0$
(anything smaller than the quantity $\rho(\kappa)$ in~\cite{BJR}) 
such that whp $H_n$ has a component with at least $\alpha n$ vertices.
For the particular $\kappa$
considered here, which is `rank 1', we have $\norm{T_\kappa}=\sum_k 2\theta k^2\mu_k$;
see (16.8) in~\cite{BJR}. Note that if $H_n$ contains a component with at least $\alpha n$ vertices,
then so does $\tF_{\theta}$ -- the union of the components of $F$ corresponding to these
vertices of $H_n$.
So, in short, if \eqref{BJR1} and \eqref{BJR2} hold, then $\tF_{\theta}$
will have a giant component (whp) if (and, one can check, only if) $\sum_k 2\theta k^2\mu_k>1$.
Moreover, it is not hard to check that these conclusions remain true if we delete 
some subset of the components of $F_n$, and adjust $\mu$, as long as \eqref{BJR1}
holds for the new graph and $\Theta(n)$ components remain; this is because
\eqref{vspace} still holds, and the kernel is still graphical.

We shall apply the observations above with initial graph $F=F_n=G_{\tcx n}^{\cR}$,
where $\cR$ is some $\ell$-vertex size rule. By \eqref{Nktcx},
the condition \eqref{BJR1} holds with $\mu_k=\rho_k(\tcx)/k$.
Furthermore, as noted after \eqref{Nktcx}, we have $\sum_k \rho_k(\tcx)=1$, which gives
\eqref{BJR2}. Finally, note that
\begin{equation}\label{infnorm}
 \norm{T_\kappa} = \sum_k 2\theta k\rho_k(\tcx) =2\theta s(\tcx) = \infty,
\end{equation}
since $s(t)=\sum_k \rho_k(t)$ diverges at $t=\tcx$. So far this tells us
only that if we run any size rule up to time $t=\tcx^{\cR}$ and then
switch to adding uniformly random edges, after any constant times $n$
further edges a giant component will emerge. The key point is that variants
of this argument can be used to study the further evolution of $G_i^{\cR}$
for suitable rules $\cR$. A related approach was taken in~\cite{SpencerWormald2007} and~\cite{JansonSpencer2010}.
\begin{theorem}\label{conj_bs}
Let $\cR$ be a bounded-size $\ell$-vertex rule. Then the conclusion of Conjecture~\ref{conj}
holds for $\cR$; in particular, $\tc^{\cR}=\tcx^{\cR}$, and moreover for any $\eps>0$
there is an $\alpha>0$ such that whp $L_1(G^{\cR}_{(\tcx+\eps)n})\ge \alpha n$.
\end{theorem}
Note that this result was proved for some bounded-size $4$-vertex  
rules (ones in which either $v_1v_2$ or $v_3v_4$ is added) already by 
Spencer and Wormald~\cite{SpencerWormald2007}.
\begin{proof}
By definition of bounded-size rules, there is a constant $B$ such that $\cR$ treats all components
of size greater than $B$ in the same way.
Consider the graph $G^{\cR}_{\tcx n}$ generated by the rule after $\tcx n$ steps. Let
$W$ be the set of vertices of this graph in components of size greater than $B$, and let
$F=F_n$ be the subgraph of $G^{\cR}_{\tcx n}$ induced by $W$. 
Noting that $s(\tcx)=\sum_k k\rho_k(\tcx)=\infty>B$, we have $\rho_k(\tcx)>0$ for some $k>B$,
and it follows that for some constant $\beta>0$, we have $|W|\ge\beta n$ whp. From
now on we assume that this is the case. In all subsequent steps of our original
process $G_i^{\cR}$, every vertex of $W$ is in a component of size greater than $B$.
Fix $\eps>0$. Let us call a step \emph{good} if in this step all $\ell$ selected vertices are in $W$.
Then each step is good with probability at least $\beta^\ell$, and it follows
that whp at least $\theta n$ of the next $\eps n$ steps are good, where $\theta=\eps \beta^\ell/2$ is a
positive constant. 
Again using the definition of a bounded-size rule, in each good step at least one edge is added and 
by symmetry it is chosen uniformly at random from all possible edges with ends in $W$. 
It follows that we may couple $G^{\cR}_{(\tcx+\eps)n}$ and $\tF_{\theta}$ so that whp
the former contains the latter as a subgraph. But $F$ satisfies the assumptions
above with $\mu_k=\rho_k(\tcx)/k$ for $k> B$ and $\mu_k=0$ for $k\le B$.
Since the sum in \eqref{infnorm} remains infinite after removing the first $B$ terms,
Theorem~3.1 of~\cite{BJR} and the discussion above imply that for some positive
$\alpha$, whp $\tF_{\theta}$ contains a component with at least $\alpha n$ vertices.
\end{proof}

Our next result concerns a different generalization of the Bohman--Frieze process~\cite{BF2001}. 
Let us call an Achlioptas rule $\cR$ \emph{take-it-or-leave-it} if, when presented
with a choice of two edges $e_1$ and $e_2$, the rule decides which to select
depending only on the current graph and on $e_1$. In other words, the rule
first sees $e_1$ and must decide whether to take this edge of not; if not, 
it selects the uniformly random edge $e_2$. Bounded-size rules of this 
type were studied, for example, by Bohman and Kravitz~\cite{BohmanKravitz2006}; here 
we do not assume that the rule is bounded-size. 
\begin{theorem}
Let $\cR$ be a take-it-or-leave-it size rule. Then the conclusions of Conjecture~\ref{conj}
and Theorem~\ref{conj_bs} hold for $\cR$. 
\end{theorem}
\begin{proof}
Consider the process $(G^{\cR}_{\tcx n+i})_{i\ge 0}$, i.e., our Achlioptas process
started at step $\tcx n$. As above, set $F=F_n=G^{\cR}_{\tcx n}$.
Since $\cR$ is a take-it-or-leave-it rule, the further evolution may be described
as follows. Let $\cL_1$ and $\cL_2$ be independent lists of independent (potential) edges
each chosen uniformly at random from all $\binom{n}{2}$ possibilities.
In step $i$ of our process (step $\tcx n+i$ of the original), take for $e_1$ the
$i$th element of $\cL_1$. The rule now decides whether to add this edge to the current
graph. If not, take for $e_2$ the \emph{next} edge from $\cL_2$, and add that.
Thus, the $j$th time that the rule declines the first edge, we take the $j$th edge from $\cL_2$.

Since the edges in $\cL_2$ are uniformly random, the discussion before Theorem~\ref{conj_bs}
shows that for any constant $\delta>0$, whp the first $\delta n$ edges from $\cL_2$
will, when added to $F=G^{\cR}_{\tcx n}$, be enough to form a giant component.
Fix $\eps>0$, and define as above a graph $H_n$ whose vertices are the components
of $F$, with edges corresponding to the first $\eps n$ edges from $\cL_1$.
As noted before, this graph $H_n$ 
may be viewed as an instance of the model studied in~\cite{BJR}, and there is some
$\alpha>0$ such that whp
$H_n$ has a component with at least $\alpha n$ vertices. Furthermore,
by the stability result~\cite[Theorem 3.9]{BJR}, there is some $\delta>0$
such that whp $H_n$ has the property that deleting any $\delta n$ edges
still leaves a component with at least $\alpha n/2$ vertices
of $H_n$. Hence, whp $\cL_1$ has the property that if we add any subset
of at least $(\eps-\delta)n$ of the first $\eps n$ edges to $F$,
we will create a component of size at least $\alpha n/2$, and 
whp $\cL_2$ has the property that adding its first $\delta n$
edges to $F$ creates a component of size at least some constant
times $n$. But when both properties hold, then whatever the rule does,
$G^{\cR}_{(\tcx+\eps)n}$ will have a giant component.
\end{proof}

For our final result, let us call an Achlioptas rule \emph{large-biased} if 
there exists some constant $B$ such that if both endvertices of $e_1$
are isolated vertices (components of size one) and both endvertices of $e_2$ are in components 
of size greater than $B$, then the rule will select $e_2$. 
Perhaps the most interesting examples of such 
rules are the \emph{reverse product rule}, where we select the (a if there is a tie) edge 
maximizing the product of the sizes of the components containing its endvertices,
or the \emph{reverse sum rule}, defined similarly but with product replaced by sum.
Perhaps surprisingly (given the difficulty of analyzing the usual product rule),
we can prove Conjecture~\ref{conj} for such rules.
\begin{theorem}
Let $\cR$ be a large-biased size rule.
Then the conclusions of Conjecture~\ref{conj}
and Theorem~\ref{conj_bs} hold for $\cR$.
\end{theorem}
\begin{proof}
The proof is very similar to that of Theorem~\ref{conj_bs}. Indeed, as usual
we start from $F=G_{\tcx n}^{\cR}$. As before, let $W$ be the set of vertices
of $F$ in components of size greater than $B$.
Call a subsequent step \emph{good} if $e_1$ joins two vertices in components of size one 
and $e_2$ joins two vertices in $W$. Since there are whp at least some
constant times $n$ isolated vertices in $G^{\cR}_{(\tcx +1)n}$, and (as before), $W$
whp has size at least a constant times $n$, off an event of small probability
the conditional probability (given the history) that the next step is good
is always at least some positive constant. Furthermore, when a step is good,
the added edge is uniformly random among all possible edges inside $W$. The
remainder of the argument is as for Theorem~\ref{conj_bs}; we omit the details.
\end{proof}

The results above all illustrate the idea that if we can find a reasonable 
number of uniformly random edges among the edges selected by our process, then 
the process will be `well behaved' (will have $\tc=\tcx$). 
This approach can be used to prove Conjecture~\ref{conj} for other special 
classes of size rules, but it seems that additional ideas are needed 
for the general case.

\subsection{Examples of delayed percolation}\label{sec:delay}
Having given several partial results supporting our belief in 
Conjecture~\ref{conj}, in this section we show that the conjecture cannot be 
extended to arbitrary $\ell$-vertex rules. More concretely, we give examples of 
simple rules that can delay the appearance of linear size components for 
$\Omega(n)$ steps beyond the point where the susceptibility diverges. The rules 
we use behave like size rules almost all the time.

We start by introducing the \emph{$r$-sum rule} $\cS_r$, which is a $2r$-vertex 
size rule. Given vertices $(v_{1}, \ldots, v_{\ell})$ and the corresponding 
list of component sizes $(c_1, \ldots, c_{\ell})$, the $r$-sum rule adds the 
pair $v_{2j-1}v_{2j}$ with the (smallest, if there are ties) $j \in [r]$ that 
minimizes the sum $c_{2j-1}+c_{2j}$ of the component sizes. 
Recall the definition of $F_{i}^{\cR}$ given in Section~\ref{sec:evo}: 
informally it denotes the graph that we obtain by starting with the 
initial graph $F_{0}^{\cR}=F$ and then following $i$ steps of an Achlioptas 
process using the rule $\cR$ (to decide which edges to add in each step). 
Intuitively, the next lemma states that the $r$-sum rule does not substantially 
change (uniform) polynomial tails for $N_{\ge k}$ during some $\delta n$ steps 
(here we use $\cS_r$ for concreteness; other size rules exhibit similar 
behaviour). 
\begin{lemma}\label{lem:dg}
Let $F$ be a graph on $n$ vertices. Suppose there are $x,C>0$ and 
$K=K(n) \ge 1$ such that for all $1 \le k \le K$ we have
\begin{equation}\label{eq:dg:hyp}
N_{\ge k}(F) \le C k^{-x}n.
\end{equation}
Given $r \ge 1+1/x$ there exists 
$\delta = \delta(x,C,r)>0$ such if $n$ is large enough then, 
with probability at least $1-n^{-99}$, for 
all $1 \le k \le {K'}=\min\{K,n^{1/[2(1+x)]}\}$ we have 
\begin{equation}\label{eq:dg:con}
N_{\ge k}(F_{\delta n}^{\cS_r}) \le 2C k^{-x}n . 
\end{equation}
\end{lemma}
\begin{proof}
Set $\delta=2^{-[(2+x)r+3]}C^{-(r-1)}$. 
Let $\cE_{i',k'}$ denote the event 
that for all $0 \le i \le i'$ and $1 \le k \le k'$ we have 
\begin{equation}\label{eq:dg:ind}
N_{\ge k}(F_{i}^{\cS_r}) \le 2C k^{-x} n .
\end{equation}
Observe that it suffices to show that $\cE_{\delta n,{K'}}$ fails with 
probability at most $n^{-99}$. For $k \le {K'}$ let $X_{k,i}$ denote 
the indicator function of the event 
$N_{\ge k}(F_{i}^{\cS_r}) \neq N_{\ge k}(F_{i-1}^{\cS_r})$. Set 
$X_{k} = \sum_{1 \le i \le \delta n} X_{k,i}$ and 
$Y_k = \sum_{1 \le i \le \delta n} Y_{k,i}$, where
\[
Y_{k,i} \; = \;	\begin{cases}
		X_{k,i}, & ~~\text{if $\cE_{i-1,k-1}$ holds} ,\\
		0, & ~~\text{otherwise} .
	\end{cases}
\]
Note that in each step a new component of size at least $k$ is only created by 
$\cS_r$ if in each pair $v_{2j-1}v_{2j}$ at least one vertex is in a component 
of size at least $\ceil{k/2}$. So, whenever $\cE_{i-1,k-1}$ holds, using 
\eqref{eq:dg:ind} and $r \ge 1+1/x$, we see that the probability that 
$X_{k,i}=1$ is at most 
\[
\left(\frac{2 N_{\ge \ceil{k/2}}(F_{i-1}^{\cS_r})}{n}\right)^{r} \le \left(\frac{4C}{(k/2)^{x}}\right)^{r}
 = \frac{\left(2^{2+x}C\right)^r}{k^{rx}}
 \le \frac{\left(2^{2+x}C\right)^r}{k^{1+x}} = \xi_{k} .
\]
Since $Y_{k,i}=0$ whenever $\cE_{i-1,k-1}$ fails, it follows that $Y_k$ is 
stochastically dominated by a binomial random variable with $\delta n$ trials 
and success probability $\xi_{k}$. 
Note that \eqref{eq:dg:hyp} implies $C \ge 1$. 
Now, using $k \le {K'}$ we have $\delta n \xi_k \ge C/8\cdot n^{1/2} \ge 600 \log n$ 
for $n \ge n_0$, so standard Chernoff bounds yield 
\begin{equation}\label{eq:dg:CB}
\Pr(Y_k \ge 2\delta n \xi_{k}) \le e^{-\delta n \xi_k/3} \le n^{-200} . 
\end{equation}
Next we claim that $\cE_{\delta n,k-1}$ and $Y_{k} < 2\delta n \xi_{k}$ 
together imply $\cE_{\delta n,k}$, so that 
$\Pr(\neg \cE_{\delta n,k}) \le \Pr(\neg \cE_{\delta n,k-1}) + \Pr(Y_k \ge 2\delta n \xi_{k})$. 
Indeed, by monotonicity $\cE_{\delta n,k-1}$ implies $\cE_{i,k-1}$ for every 
$0 \le i \le \delta n$, so $Y_k=X_k$. Now, since $N_{\ge k}$ increases by at 
most $2k$ per step, by choice of $\delta$ it follows that 
\[
N_{\ge k}(F_{\delta n}^{\cS_r})-N_{\ge k}(F_{0}^{\cS_r}) \le 2k Y_k \le 4\delta \left(2^{2+x}C\right)^r k^{-x} n \le C k^{-x} n , 
\]
which together with $N_{\ge k}(F_{0}^{\cS_r}) = N_{\ge k}(F)\le C k^{-x} n$ 
implies $N_{\ge k}(F_{\delta n}^{\cS_r}) \le 2C k^{-x} n$, as claimed. 
Iterating the above argument for $k \le {K'}$ and noting that 
$\cE_{\delta n,1}$ always holds due to $C \ge 1$, 
using \eqref{eq:dg:CB} we obtain 
\[
\Pr(\neg \cE_{\delta n,{K'}}) \le \sum_{2 \le k \le {K'}} \Pr(Y_{k} \ge 2\delta n \xi_{k}) \le n^{-99} ,
\]
and the proof is complete. 
\end{proof}
Let $\cD_{r}$ denote the rule which always adds the pair $v_1v_2$ during the
first $n/2$ steps (corresponding to an Erd\H os--R\'enyi evolution with 
$\ell = 2$); afterwards it `switches' and uses the $r$-sum rule $\cS_r$. The 
point is that many properties of the `critical' Erd\H os--R\'enyi random graph 
$G_{n,n/2}$ are well known: there exist constants $C,\alpha > 0$ and a function 
$K=K(n)$ with $K \to \infty $ as $n \to \infty $ such that whp 
$S(G_{n,n/2}) \ge n^{\alpha}$ and $N_{\ge k}(G_{n,n/2}) \le C k^{-1/2}n$ for 
all $1 \le k \le K$. So, by conditioning on these properties and then using 
Lemma~\ref{lem:dg}, we immediately deduce the main result of this section. 
Indeed, using the rule $\cD_{r}$ for $r \ge 3$ we whp have diverging 
susceptibility after $n/2$ steps, but in $\delta n$ subsequent steps whp no 
linear size components appear (in fact, in this case $\tcx=1/2<\tc$ holds). 
\begin{corollary}\label{cor:delayI}
For every $r \ge 3$ there exists $\delta=\delta(r)>0$ such that we have whp 
$S(G_{n/2}^{\cD_{r}}) = \omega(1)$ and $L_1(G_{n/2+\delta n}^{\cD_{r}})=o(n)$. 
\end{corollary} 
Alternatively, using essentially the same line of reasoning, we obtain a 
similar result by switching after the first step where the susceptibility is at 
least $L=L(n)=\omega(1)$, for $L$ not too large. 
Furthermore, we can replace $\cS_r$ by other suitable size rules. For example, 
the rule $\cM_\ell$, which always connects two vertices with the two smallest 
component sizes $c_j$, satisfies an analogue of Lemma~\ref{lem:dg} for 
$\ell \ge 2+1/x$. So the rule $\cC_{\ell}$, which switches from an 
Erd\H os--R\'enyi evolution (always adding $v_1v_2$) to $\cM_\ell$ after 
$n/2$ steps, yields another example with $\tcx<\tc$. 
\begin{corollary}\label{cor:delayC}
For every $\ell \ge 4$ there exists $\delta=\delta(\ell)>0$ such that we have whp 
$S(G_{n/2}^{\cC_{\ell}}) = \omega(1)$ and $L_1(G_{n/2+\delta n}^{\cC_{\ell}})=o(n)$. 
\end{corollary} 
Note that the examples given in Corollary~\ref{cor:delayI} and~\ref{cor:delayC} 
always behave like size rules except that once between two steps they change the 
rule used (by only querying natural parameters such as the number of vertices 
and steps, or the susceptibility). So, one can argue that Conjecture~\ref{conj} 
already fails for a rather restricted superset of size rules.

\bibliographystyle{plain}

\begin{thebibliography}{10}

\bibitem{Science2009}
D.~Achlioptas, R.M.~D'Souza, and J.~Spencer. 
\newblock Explosive percolation in random networks. 
\newblock {\em Science} {\bf 323} (2009), 1453--1455.

\bibitem{AB1987}
M.~Aizenman and D.J.~Barsky
\newblock Sharpness of the phase transition in percolation models. 
\newblock {\em Comm.\ Math.\ Phys.} {\bf 108} (1987), 489--526.

\bibitem{BBW2011}
S.~Bhamidi, A.~Budhiraja, and X.~Wang.
\newblock {B}ohman--{F}rieze processes at criticality and emergence of the
  giant component.
\newblock Preprint, 2011. \texttt{arXiv:1106.1022}. 

\bibitem{BF2001}
T.~Bohman and A.~Frieze. 
\newblock Avoiding a giant component. 
\newblock {\em Random Struct.\ Alg.} {\bf 19} (2001), 75--85.

\bibitem{BFKLS2011}
T.~Bohman, A.~Frieze, M.~Krivelevich, P.-S.~Loh, and B.~Sudakov.
\newblock Ramsey games with giants
\newblock {\em Random Struct.\ Alg.} {\bf 38} (2011), 1--32.

\bibitem{BohmanKravitz2006}
T.~Bohman and D.~Kravitz. 
\newblock Creating a giant component. 
\newblock {\em Combin.\ Probab.\ Comput.} {\bf 15} (2006), 489--511.

\bibitem{BJR}
B.~Bollob{\'a}s, S.~Janson, and O.~Riordan.
\newblock The phase transition in inhomogeneous random graphs.
\newblock {\em Random Struct.\ Alg.} {\bf 31} (2007), 3--122.

\bibitem{BBRG}
B.~Bollob{\'a}s,
\newblock {\em Random Graphs}, 2nd ed., Cambridge University Press (2001). 

\bibitem{BR-RG}
B.~Bollob{\'a}s and O.~Riordan.
\newblock Random graphs and branching processes.
\newblock In {\em Handbook of large-scale random networks}, 
\newblock Bolyai Soc.\ Math.\ Stud {\bf 18} (2009), pp.~15--115.

\bibitem{Science2011J}
S.~Janson.
\newblock Networking -- {S}moothly does it.
\newblock {\em Science} {\bf 333} (2011), 298--299.

\bibitem{JKLP1993}
S.~Janson, D.E.~Knuth, T.~{\L}uczak and B.~Pittel. 
\newblock The birth of the giant component.
\newblock {\em Random Struct.\ Alg.} {\bf 4} (1993), 231--358.

\bibitem{JansonSpencer2010}
S.~Janson and J.~Spencer. 
\newblock {Phase transitions for modified {E}rd\H{o}s-{R}\'enyi processes}. 
\newblock {\em Ark.\ Math.}, to appear. \texttt{arXiv:1005.4494}. 

\bibitem{KPS2011}
M.~Kang, W.~Perkins, and J.~Spencer.
\newblock The {B}ohman--{F}rieze process near criticality.
\newblock {\em Random Struct.\ Alg.}, to appear. \texttt{arXiv:1106.0484}. 

\bibitem{Karp}
R.M.~Karp. 
\newblock The transitive closure of a random digraph.
\newblock {\em Random Struct.\ Alg.} {\bf 1} (1991), 73--93.

\bibitem{McDiadmid1989}
C.~McDiarmid. 
\newblock On the method of bounded differences. 
\newblock In {\em Surveys in {C}ombinatorics} ({N}orwich, 1989), London Math.\ Soc.\ Lecture Note Ser., vol. 141, pp. 148--188. Cambridge Univ. Press, Cambridge, 1989.

\bibitem{Menshikov1986}
M.V.~Men'shikov,
\newblock Coincidence of critical points in percolation problems. 
\newblock {\em Dokl.\ Akad.\ Nauk SSSR} {\bf 288} (1986), 1308--1311.

\bibitem{arxiv}
O.~Riordan and L.~Warnke.
\newblock Achlioptas processes can be nonconvergent. Preprint, 2011.
\texttt{arXiv:1111.6179}. 

\bibitem{AAP2011}
O.~Riordan and L.~Warnke.
\newblock Achlioptas process phase transitions are continuous.
\newblock {\em Ann.\ Appl.\ Probab.}, to appear. \texttt{arXiv:1102.5306}. 

\bibitem{RW2011}
O.~Riordan and L.~Warnke.
\newblock Convergence of Achlioptas processes via differential equations with unique solutions. Preprint, 2011. \texttt{arXiv:1111.6179}. 

\bibitem{Science2011}
O.~Riordan and L.~Warnke.
\newblock Explosive percolation is continuous.
\newblock {\em Science} {\bf 333} (2011), 322--324.

\bibitem{SpencerWormald2007}
J.~Spencer and N.C.~Wormald. 
\newblock Birth control for giants. 
\newblock {\em Combinatorica} {\bf 27} (2007), 587--628.

\bibitem{Wormald1995DEM}
N.C.~Wormald. 
\newblock Differential equations for random processes and random graphs.
\newblock {\em Ann.\ Appl.\ Probab.} {\bf 5} (1995), 1217--1235.

\bibitem{Wormald1999DEM}
N.C.~Wormald. 
\newblock The differential equation method for random graph processes and
  greedy algorithms. 
\newblock In {\em Lectures on approximation and randomized algorithms}, pages
  73--155. PWN, Warsaw, 1999.

\end{thebibliography}

\begin{appendix}
\section{Appendix}\label{appendix}
In this appendix we show that, as long as an Achlioptas process contains only 
small components, it will have a very simple cycle structure: most components 
will be trees, some will be unicyclic, and there will (whp) be no `complex' 
components, i.e., ones containing more than one cycle. In fact, we prove the 
result for the more general class of $\ell$-vertex rules. However, here we need 
an additional assumption: in each round the set of added edges is a forest. 
Allowing slightly greater generality, we call a rule \emph{acyclic} if the edges 
between $v_1,\ldots,v_\ell$ added in a single step correspond to a forest on 
$1,\ldots,\ell$. This in particular holds if in each step at most two edges are 
added. Note that such an assumption is in fact necessary for $\ell \ge 3$, 
since always connecting all vertices in each step quickly creates many cycles 
and complex components. 
\begin{lemma}\label{lem:Simple}
Let $\ell \ge 2$ and let $\cR$ be an acyclic $\ell$-vertex rule. 
Given $0 < \delta < 1/4$ and $U=U(n)$, suppose that for $n \ge n_0(\ell,\delta)$ 
we have $1 \le U \le n^{1/4-\delta}$. 
For $n \ge n_0(\ell,\delta)$, with probability at least $1-n^{-\delta/2}$ the 
following holds for every $0 \le i \le n^{1+\delta}$: in $G_i^{\cR}$ there are 
no components of size at most $U$ which contain at least two cycles, and the 
number of vertices in components of size at most $U$ with exactly one cycle is 
at most $U^2n^{2\delta}$. 
\end{lemma}
\begin{proof}
Set $m=n^{1+\delta}$. 
Let $\cB_{1,i}$ denote the event that in step $i$ one of the following happens: 
(a) at least three of the $\ell$ randomly chosen vertices are contained in the 
same component of size at most $U$, or (b) there are randomly chosen vertices 
$w_1,w_2,w_3,w_4$ and two components $C_1,C_2$ of size at most $U$ such that 
$w_1,w_2 \in C_1$ and $w_3,w_4 \in C_2$. It is easy to see that $\cB_{1,i}$ 
holds with probability at most $\ell^3U^2/n^2+\ell^4 U^2/n^2$. So, denoting by 
$\cB_{1}$ the event that $\cB_{1,i}$ holds for some $i \le m$, we see that 
\[
\Pr(\neg \cB_{1}) \le m \cdot 2\ell^4 U^2/n^2 \le 2\ell^4 n^{-1/2-\delta} .
\]

Let $S_{i,U}$ denote the set of vertices of $G_i^{\cR}$ which are in components 
of size at most $U$ containing exactly one cycle, and let $\cB_{2}$ be the 
event that $S_{i,U}$ contains at least $2\ell^3 U m/n$ components for some 
$i \le m$. Then $\neg\cB_{2}$ implies $|S_{i,U}| \le 2\ell^3 U^2 m/n=R$ for 
every $i \le m$, where $R \le U^2 n^{2\delta}$ for $n \ge n_0(\ell,\delta)$. 
Since in each step the number of components in $S_{i,U}$ changes by at most 
$\ell$, when $\cB_2$ holds there are at least $2\ell^2 U m/n$ steps in which 
the number of components in $S_{i,U}$ increases. This can only happen if at 
least two randomly chosen vertices are in the same component of size at most 
$U$. Since in each step this has probability at most $\ell^2U/n$, the expected 
number of such steps is bounded by $\lambda=\ell^2 U m/n$. Using standard 
Chernoff bounds (and stochastic domination) it follows that 
\[
\Pr(\cB_{2}) \le e^{-\lambda/3} \le e^{-n^{\delta}} . 
\]

Let $\cB_{3}$ denote the event that in some $G_i^{\cR}$ with $i \le m$ there 
exists a component of size at most $U$ which contains at least two cycles. In 
each step where $\cB_{1,i}$ fails, note that a new complex component of size at 
most $U$ can only be created if (a) at least two randomly chosen vertices are 
in $S_{i-1,U}$ or (b) one randomly chosen vertex lies in $S_{i-1,U}$, and two 
randomly chosen vertices are in the same tree component of size at most $U$. 
So, by considering the probability that this happens for some $i \le m$ 
(assuming $|S_{i-1,U}| \le R$), we see that 
\[
\Pr(\neg \cB_{1} \cap \neg\cB_{2} \cap \cB_{3}) \le m \cdot (\ell^2R^2/n^2 + \ell^3RU/n^2) \le 5 \ell^8U^4m^3/n^4 \le 5\ell^8n^{-\delta} , 
\]
completing the proof for $n \ge n_0(\delta,\ell)$.  
\end{proof}
Theorem~\ref{thm:main} tells us that for size rules, for any fixed $t<\tcx$, 
the largest component of $G^{\cR}_{tn}$ whp has size at most $O(\log n)$. 
Taking $U=(\log n)^2$, say, we see that if $\cR$ is acyclic, then whp there are 
no complex components, and at most $n^{o(1)}$ vertices on cyclic components -- 
in other words, almost all components are trees. (We have not tried to optimize 
the bound here -- the method actually gives some power of $\log n$.)
\end{appendix}

\end{document}